\documentclass[12pt]{amsart}
\usepackage{amsmath, amsthm, graphicx, amssymb,verbatim}
\usepackage[margin=1.0in]{geometry}
\ifx\JPicScale\undefined\def\JPicScale{0.8}\fi
\unitlength \JPicScale mm

{\theoremstyle{plain}
   
   \newtheorem{theorem}{Theorem}[section]
   \newtheorem{proposition}[theorem]{Proposition}     
   \newtheorem{lemma}[theorem]{Lemma}
   
   \newtheorem{corollary}[theorem]{Corollary}

}
{\theoremstyle{definition}

   \newtheorem{definition}[theorem]{Definition}
   \newtheorem{remark}[theorem]{Remark}

}
\newcommand{\quotient}{/\hspace{-1.2mm}/}
\numberwithin{theorem}{section}

\begin{document}
\title{GIT Compactifications of $\mathcal{M}_{0,n}$ from Conics}
\author{Noah Giansiracusa, Matthew Simpson}
\maketitle

\begin{abstract}
We study GIT quotients parametrizing n-pointed conics that generalize the GIT quotients $(\mathbb{P}^1)^n/\hspace{-0.8mm}/\mathtt{SL}_2$. Our main result is that $\overline{\mathcal{M}}_{0,n}$ admits a morphism to each such GIT quotient, analogous to the well-known result of Kapranov for the simpler $(\mathbb{P}^1)^n$ quotients. Moreover, these morphisms factor through Hassett's moduli space of weighted pointed rational curves, where the weight data comes from the GIT linearization data.
\end{abstract}

\tableofcontents

\section{Introduction and main results}

Inspired by recent work of the second author \cite{Simp}, we study a family of GIT quotients parametrizing $n$-pointed conics that generalize the GIT quotients $(\mathbb{P}^1)^n\quotient\mathtt{SL}_2$.  These latter quotients compactify the moduli space $\mathcal{M}_{0,n}$ of nonsingular $n$-pointed rational curves by allowing points to collide as long as their \emph{weight} (a number assigned to each point when choosing a linearization for the group action) is not too much.  For the GIT quotients that we investigate, denoted $\mathtt{Con}(n)\quotient\mathtt{SL}_3$, the compactification allows a certain number of points to overlap based on their weights, but if too many points collide then the nonsingular conic degenerates into a nodal conic.  Up to isomorphism nonsingular and nodal conics are a $\mathbb{P}^1$ and a pair of intersecting $\mathbb{P}^1$s, respectively, so the spaces $\mathtt{Con}(n)\quotient\mathtt{SL}_3$ can be viewed as intermediate compactifications between $(\mathbb{P}^1)^n\quotient\mathtt{SL}_2$ and the well-known Deligne-Mumford-Knudsen compactification $\overline{\mathcal{M}}_{0,n}$ \cite{Knud}.  

The following theorem characterizes GIT stability for pointed conics, generalizing a result of the second author (Theorem 3.1.5 in \cite{Simp}) which describes stability in the special case of $S_n$ invariant weights.

\begin{theorem} \label{thm:stability} Let $(\gamma,c_1,\ldots,c_n)$ specify an ample fractional line bundle on the space of $n$-pointed conics $\mathtt{Con}(n)\subset \mathbb{P}(\mathtt{Sym}^2(V^*))\times(\mathbb{P}(V))^n$, $V=\mathbb{C}^3$, linearized for the natural action of $\mathtt{SL}(V)$.  If $c := c_1 +\cdots + c_n$ is the total point weight then: 
\begin{itemize}
	\item all non-reduced conics are unstable
	\item a nodal conic is semistable iff 
	\begin{enumerate} 
		\item the weight of marked points at any smooth point is $\le \frac{c+\gamma}{3}$ 
		\item the weight of marked points at the node is $\le c - 2(\frac{c+\gamma}{3})$, and 
		\item the weight on each component is $\le \frac{2c-\gamma}{3}$; equivalently, the weight on each component away from the node is $\ge \frac{c+\gamma}{3}$
	\end{enumerate}
	\item a nonsingular conic is semistable iff the weight at each point is $\le \mathtt{min}\{\frac{c+\gamma}{3},\frac{c}{2}\}$
\end{itemize}
In particular, if $\gamma > \frac{c}{2}$ then nodal conics are unstable.  Stability is characterized by the corresponding inequalities being replaced by strict inequalities. 
\end{theorem}

A variation of GIT perspective will be useful in our investigations.  When a reductive group $G$ acts on a variety the space of linearized fractional polarizations forms a cone called the \emph{$G$-ample cone}, and inside it sits the \emph{$G$-effective cone} which is defined as the set of linearizations for which the semistable locus is nonempty.  The $G$-effective cone admits a finite wall and chamber decomposition such that on each open chamber the GIT quotient is constant and when a wall is crossed the quotient undergoes a birational modification (see \cite{Thad} and \cite{DH}).  In some cases this cone admits a natural cross-section so that the (closure of) the space of linearizations can be identified with a certain polytope which we call the \emph{linearization polytope}.  

For example, the GIT quotients $(\mathbb{P}^m)^n\quotient\mathtt{SL}_{m+1}$ parametrizing configurations of $n$ points in $m$-dimensional projective space have $\mathtt{SL}_{m+1}$-ample cone $\mathbb{Q}_{>0}^n$ because $\text{Pic}((\mathbb{P}^m)^n) \cong \mathbb{Z}^n$ and each line bundle admits a unique linearization.   A vector $\vec{c}=(c_1,\ldots,c_n)\in\mathbb{Q}_{>0}^n$ assigns a positive rational weight to each point, and a configuration is semistable if and only if the total weight lying in any proper linear subspace $W\subset\mathbb{P}^m$ is at most $\frac{\text{dim }W+1}{m+1}\cdot \sum_1^n c_i$ (see, e.g., Example 3.3.21 in \cite{DH}).  Multiplying $\vec{c}$ by a positive constant does not affect stability so one can use the normalization $\sum_1^n c_i = m+1$.  The semistable locus is then non-empty precisely when $\text{max}\{c_i\} \le 1$ so the linearization polytope is a hypersimplex \[\Delta(m+1,n) = \{\vec{c}\in\mathbb{Q}^{n}~|~0\le c_i \le 1, \sum_{i=1}^n c_i = m+1\}\] with walls of the form $\sum_{i\in I} c_i=k$ for $I\subset\{1,\ldots,n\}$ and $1 \le k \le m$.

In particular, for points on the line ($m=1$) we have $\Delta(2,n)$ with walls $\sum c_i = 1$, and for points in the plane ($m=2$) we have $\Delta(3,n)$ with walls $\sum c_i =1$ and $\sum c_i = 2$.  A consequence of Theorem \ref{thm:stability} is that for the space of $n$-pointed conics the effective linearizations form a 1-parameter family of hypersimplices that interpolate these two cases. 

\begin{corollary} \label{cor:chamber} The $\mathtt{SL}_3$-effective cone for $\mathtt{Con}(n)$ induced from that of the ambient $\mathbb{P}^5\times(\mathbb{P}^2)^n$ is subdivided by the hyperplane $\gamma = \frac{c}{2}$ into two subcones: $\gamma \le \frac{c}{2}$ for which semistable nodal conics occur, and $\gamma > \frac{c}{2}$ for which singular conics are unstable.  When normalizing with cross-sections $\gamma + c = 3$ on the former and $c=2$ on the latter the linearization polytopes for fixed $\gamma$ are $\Delta(3-\gamma,n)$ with walls $\sum c_i =1$ and $\sum c_i=2$ if $0\le\gamma\le 1$, and $\Delta(2,n)$ with walls $\sum c_i = 1$ if $\gamma \ge 1$.  These cross-sections meet at the codimension 2 linear subspace $\gamma=1,c=2$ and only miss the ray $\gamma \ne 0$, $c=0$.
\end{corollary}

\begin{remark} If $\gamma > \frac{c}{2}$ then, as we discuss later, $\mathtt{Con}(n)\quotient_{(\gamma,\vec{c})}\mathtt{SL}_3 \cong (\mathbb{P}^1)^n\quotient_{\vec{c}}~\mathtt{SL}_2$ so to get novel compactifications we can restrict $\gamma$ to the interval $[0,\frac{c}{2}]$.  In this case only one normalization is necessary (namely $\gamma+c=3$) and the linearization polytope is $\Delta(3,n+1)$ with walls $\sum c_i=1$ and $\sum c_i = 2$ which are ``vertical'' in the sense that they are independent of $\gamma$.
\end{remark}

Another interesting family of compactifications is provided by Hassett's moduli spaces of stable weighted pointed curves \cite{Hass}.  Recall that for a weight vector $\vec{c}\in[0,1]^n$ the space $\overline{\mathcal{M}}_{0,\vec{c}}$ parametrizes nodal rational curves with marked points $p_i$ avoiding the nodes such that on any component $C$ we have $\sum_{p_i \in C} c_i + \delta_C > 2$, where $\delta_C$ is the number of nodes on $C$.  In particular, if $c_i=1$ for $1\le i \le n$ then $\overline{\mathcal{M}}_{0,\vec{c}}\cong \overline{\mathcal{M}}_{0,n}$ so these spaces can also be viewed as intermediate compactifications of $\mathcal{M}_{0,n}\subset \overline{\mathcal{M}}_{0,n}$.  The main result of this paper (generalizing the $S_n$ invariant result Theorem 3.2.6 in \cite{Simp}) is that these Hassett compactifications and our conic compactifications are related in the following manner.

\begin{theorem} \label{thm:contraction}
For any $(\gamma,\vec{c})\in\Delta(3,n+1)$ such that all entries of $\vec{c}$ are nonzero there is a birational contraction morphism \[\overline{\mathcal{M}}_{0,n} \rightarrow \mathtt{Con}(n)\quotient_{(\gamma,\vec{c})}\mathtt{SL}_3\] which factors through $\overline{\mathcal{M}}_{0,\vec{c}}$.
\end{theorem}

The idea is that a stable $n$-pointed rational curve becomes a stable weighted pointed curve by contracting all components which carry too little weight and then this can be further contracted to at most two components which are then embedded in the plane as a conic stable with respect to the GIT linearization corresponding to the Hassett weight data.  

\begin{remark} This theorem should be thought of as an analogue of the result of Kapranov \cite{Kapr} that $\overline{\mathcal{M}}_{0,n}$ admits a morphism to every GIT quotient $(\mathbb{P}^1)^n\quotient\mathtt{SL}_2$.  In fact, because $\mathtt{Con}(n)\quotient_{(\gamma,\vec{c})}\mathtt{SL}_3 \cong (\mathbb{P}^1)^n\quotient_{\vec{c}}~\mathtt{SL}_2$ for $\gamma > 1$, this theorem when combined with Kapranov's result shows that $\overline{\mathcal{M}}_{0,n}$ admits a morphism to every GIT quotient $\mathtt{Con}(n)\quotient\mathtt{SL}_3$ with linearization induced from the ambient product of projective spaces.
\end{remark}

The remaining sections of this paper are devoted to proving the results described in this introduction, except for the last section in which we explore some examples and further properties.\\

\textbf{Acknowledgements}.  This project was suggested by the second author's thesis adviser, Brendan Hassett, and was guided by helpful discussions with him as well as the first author's thesis advisers, Dan Abramovich and Danny Gillam.  We would like to extend our thanks to them for sharing their time and ideas with us.

\section{Numerical criteria for stability of conics}

The goal of this section is to prove Theorem \ref{thm:stability} by applying the Hilbert-Mumford numerical criterion for stability.  Let us briefly recall how this criterion works (see \cite{GIT} and \cite{News} for details).  Assume that $G$ is a reductive group, $L$ is an ample linearized line bundle, and $X$ is a projective variety over $\mathbb{C}$.  Then $L$ (or a suitable tensor power) induces a morphism $X \rightarrow \mathbb{P}^N$ to the projective space determined by its sections and the linearization gives an action of $G$ on these sections.  In particular, if $\lambda$ is a 1-parameter subgroup, i.e. a homomorphism $\mathbb{G}_m \rightarrow G$, then there is an induced linear action of $\lambda$ on $\mathbb{A}^{N+1}$.  This action can always be diagonalized: there is a basis $e_0,\ldots,e_N$ such that $\lambda(t)\cdot e_i = t^{r_i}e_i$ for some integers $r_i$ which are called the \emph{weights} of the action.  For any point $x\in X$ we look at its image in $\mathbb{P}^N$, choose a point $\widetilde{x}=\sum_{i=0}^N \widetilde{x}_i e_i$ in the affine cone $\mathbb{A}^{N+1}$ lying over it, and then acting on this by $\lambda$ gives an expression of the form $\sum_{i=0}^N t^{r_i}\widetilde{x}_ie_i$.  We also refer to $r_i$ as the weight of the coordinate $\widetilde{x}_i$.  The ``numerical criterion'' says that a point $x$ is semistable if and only if the minimum $\mu$ of the weights of its coordinates over all nonzero coordinates and all 1-parameter subgroups is $\le 0$, and $x$ is stable if and only if $\mu <0$.

We now turn to the specifics of our problem.  We are interested in the space \[\mathtt{Con}(n) \subset \mathbb{P}^5\times(\mathbb{P}^2)^n\] of $n$-pointed conics, where $\mathbb{P}^5 = \mathbb{P}(\mathtt{Sym}^2((\mathbb{C}^3)^*))$.  If we take coordinates $(x_l,y_l,z_l)$ on the $l^{\text{th}}$ copy of $\mathbb{P}^2$ and $a_{ijk}$ ($i+j+k=2$) on the $\mathbb{P}^5$ then $\mathtt{Con}(n)$ is cut out by the equations $\sum_{i+j+k=2}a_{ijk}x_l^iy_l^jz_l^k=0$ for $l=1,\ldots, n$.  The group $\mathtt{SL}_3$ acts as automorphisms on $\mathbb{P}^2$, so it acts on $\mathbb{P}^5$ by sending a conic to its image under the corresponding projective motion.  This induces an action on $\mathtt{Con}(n)$.  Concretely, if \[A = \left(\begin{array}{ccc} a_{200} & \frac{1}{2}a_{110} & \frac{1}{2}a_{101} \\ \frac{1}{2}a_{110} & a_{020} & \frac{1}{2}a_{011} \\ \frac{1}{2}a_{101} & \frac{1}{2}a_{011} & a_{002} \end{array}\right), \vec{x_l} =\left(\begin{array}{c} x_l \\ y_l \\ z_l \end{array} \right)\] are the symmetric bilinear form associated to the conic $a_{ijk}$ and coordinate vectors on $\mathbb{P}^2$, respectively, then the action is given by \[g\cdot (A,\vec{x}_1,\ldots,\vec{x}_n) = ((g^{-1})^tAg^{-1},g\vec{x}_1,\ldots,g\vec{x}_n)\] and the incidence correspondence $\mathtt{Con}(n)$ is defined by $\vec{x}^t_lA\vec{x}_l=0$ for $l=1,\ldots,n$.  The Segre embedding $\mathbb{P}^5\times(\mathbb{P}^2)^n \hookrightarrow \mathbb{P}^N$ gives $\mathtt{Con}(n)$ homogeneous coordinates of the form $(\cdots,a_{ijk}w_{I,J,K},\cdots)$, where $i+j+k=2$, $I\cup J \cup K = \{1,\ldots, n\}$ is a disjoint union, and $w_{I,J,K}$ denotes the product of $x$ coordinates indexed by $I$, $y$ coordinates by $J$, and $z$ coordinates by $K$, so that for example if $n=6$ then $w_{125,4,36}=x_1x_2z_3y_4x_5z_6$. 

Any 1-parameter subgroup of $\mathtt{SL}_3$ can be diagonalized to the form \[t\mapsto\left(\begin{array}{ccc}t^{r_1} & 0 & 0 \\0 & t^{r_2} & 0 \\0 & 0 & t^{r_3}\end{array}\right)\] where $r_1 + r_2 + r_3 = 0$ and $r_1\le r_2 \le r_3$.  Since we are interested only in stability, which is invariant under scaling all the exponents by a common factor, we can divide the $r_i$ by $r_3$ and hence assume the subgroup is of the form \[t \mapsto \left(\begin{array}{ccc}t^b & 0 & 0 \\0 & t^{-1-b} & 0 \\0 & 0 & t\end{array}\right)\] for $-2\le b \le -\frac{1}{2}$.

\begin{definition}
For each weight vector $(\gamma, c_1,\ldots,c_n)\in \mathbb{Q}_{> 0}^{n+1}$ we consider the ample line bundle on $\mathtt{Con}(n)$ defined by \[(\pi_0^*\mathcal{O}_{\mathbb{P}^5}(\gamma)\otimes \pi_1^*\mathcal{O}_{\mathbb{P}^2}(c_1)\otimes \cdots \otimes \pi_n^*\mathcal{O}_{\mathbb{P}^2}(c_n))|_{\mathtt{Con}(n)}\] where $\pi_i$ are the obvious projections, and equip it with the unique $\mathtt{SL}_3$-linearization.  For any index set $I\subset \{1,\ldots,n\}$ we write $c_I := \sum_{i\in I}c_i$; the total weight of points is denoted $c := \sum_1^n c_i$.
\end{definition}

With this setup we claim that the weight of $a_{ijk}w_{I,J,K}$ is 
\begin{equation}\label{eqn:weight}
\gamma(i(1-b) + j(b+2) - 2) + bc_I + (-1-b)c_J + c_K.
\end{equation}
Indeed, $t\cdot x_l = t^{bc_l} x_l$ so we get a term $bc_I$ coming from the $x$ variables indexed by $I$; similarly, $t \cdot y_l = t^{(-1-b)c_l}y_l$ yields $(-1-b)c_J$ and $t\cdot z_l = t^{c_l}z_l$ yields $c_K$.  In other words, $t\cdot w_{I,J,K} = t^{bc_I + (-1-b)c_J + c_K}w_{I,J,K}$.  Now to find the weight for the action on the $a_{ijk}$ factor we see from the matrix product \[\left(\begin{array}{ccc}t^{-b} & 0 & 0 \\0 & t^{1+b} & 0 \\0 & 0 & t^{-1}\end{array}\right)\left(\begin{array}{ccc} a_{200} & \frac{1}{2}a_{110} & \frac{1}{2}a_{101} \\ \frac{1}{2}a_{110} & a_{020} & \frac{1}{2}a_{011} \\ \frac{1}{2}a_{101} & \frac{1}{2}a_{011} & a_{002} \end{array}\right)\left(\begin{array}{ccc}t^{-b} & 0 & 0 \\0 & t^{1+b} & 0 \\0 & 0 & t^{-1}\end{array}\right)\] that $t\cdot a_{ijk} = t^{\gamma(i(-b) + j(1+b) - k(-1))}a_{ijk}$.  But $k = 2-(i+j)$ so $i(-b)+j(1+b)-k(-1) = i(1-b) + j(2+b) - 2$, which explains Formula (\ref{eqn:weight}).

Therefore, a pointed conic is semistable iff the minimum $\mu$ of these weights over all possible nonzero choices of $a_{ijk}w_{I,J,K}$ and all $b\in [-2,-\frac{1}{2}]$ is $\le 0$, and it is stable iff $\mu < 0$.  Recall that conics have exactly three isomorphism classes: non-reduced (a double line), nodal (a pair of intersecting lines), and non-singular (isomorphic to $\mathbb{P}^1$).  We examine these three cases in turn.

\subsubsection*{Non-reduced conics are unstable}

The orbit of any non-reduced conic contains $x^2=0$ so to show all non-reduced conics are unstable it is enough to show this one is unstable.  Here we are forced to have $i=2,j=0,k=0$.  Moreover, since all points on $x^2=0$ have vanishing $x$-coordinate we must have $I=\varnothing$, so the weight is $\gamma(-2b) + (-1-b)(c_J) + c_K$.  Setting $b=-2$ this yields $4\gamma + c_J + c_K$, which is certainly positive.

\subsubsection*{Three conditions for (semi)stability of a nodal conic}  Here we derive three necessary conditions a semistable conic must satisfy and then show they are in fact sufficient.  For the nodal conic $xy=0$ we have $i=1,j=1,k=0$ so the weight is $\gamma + bc_I + (-1-b)c_J + c_K$.  Setting $b=-2$ yields $\gamma -2c_I + c_J + c_K = \gamma + c - 3c_I$ which is minimized when $c_I$ is maximized, so $\mu$ is computed by having $I$ index all points with nonzero $x$ coordinate.  The remaining weight $c-c_I$ must lie on the line $x=0$ and $\mu \le 0$ implies $c-c_I \le \frac{2c-\gamma}{3}$ so there is $\le \frac{2c-\gamma}{3}$ weight on the component $x=0$.  Since the action of $\mathtt{SL}_3$ sends any nodal conic to $xy=0$ and either component can be sent to $x=0$ this shows that any semistable nodal conic has $\le\frac{2c-\gamma}{3}$ weight on either component, or equivalently, it has $\le c-\frac{2c-\gamma}{3}=\frac{c+\gamma}{3}$ weight on each component off the node.  

Setting $b=-\frac{1}{2}$ the weight of $xy=0$ becomes $\gamma - \frac{1}{2}c_I - \frac{1}{2}c_J + c_K = \gamma + c - \frac{3}{2}(c_I + c_J)$ so $\mu$ is computed by indexing all points with nonzero $x$ or $y$ coordinate with $I\cup J$.  The only remaining points are at the node $x=y=0$ and $\mu \le 0$ implies $c-(c_I + c_J) \le \frac{c-2\gamma}{3}$ so a semistable nodal conic has $\le\frac{c-2\gamma}{3} = c-2(\frac{c+\gamma}{3})$ weight at the node.

For the conic $xz=0$ we have $i=1,j=0,k=1$ so for $b=-\frac{1}{2}$ the weight is $-\frac{\gamma}{2} -\frac{c_I}{2}  - \frac{c_J}{2} + c_K = c - \frac{\gamma}{2} - \frac{3}{2}(c_I + c_J)$ which means again that $\mu$ is computed by indexing all points away from $x=y=0$ with $I\cup J$.  But now this point is smooth so because $\mu\le 0$ implies $c - (c_I + c_J) \le \frac{c+\gamma}{3}$ we see that the weight at any smooth point of a semistable conic is $\le \frac{c+\gamma}{3}$.\\

We next show that a nodal conic satisfying these three conditions is semistable: \[\text{min}_{a_{ijk}w_{I,J,K}\ne 0}\{\gamma(i(1-b) + j(b+2) - 2)) + bc_I + (-1-b)c_J + c_K\}\le 0\] for all $b\in[-2,-\frac{1}{2}]$.  It is enough to show that for each such conic there is a single coordinate $a_{ijk}w_{I,J,K}\ne 0$ with weight $\le 0$ at both endpoints $b=-2,b=-\frac{1}{2}$, since all other values of $b$ are linearly interpolated from these.  If we define $\text{wt}(x^iy^jz^{2-i-j}) := i(1-b) + j(b+2) - 2$ then it is easy to see that  \[-2 \le b \le -1 \Rightarrow \text{wt}(x^2) \ge \text{wt}(xy) \ge \text{wt}(xz) \ge 0 \ge \text{wt}(y^2) \ge \text{wt}(yz) \ge \text{wt}(z^2),\] \[-1 \le b \le -\frac{1}{2} \Rightarrow \text{wt}(x^2) \ge \text{wt}(xy) \ge \text{wt}(y^2) \ge 0 \ge \text{wt}(xz) \ge \text{wt}(yz) \ge \text{wt}(z^2).\]   

\emph{Case 1: $x=y=0$ is a node}.  The conic cannot be the double line $x^2=0$ so it has a monomial term with weight $\le \text{wt}(xy)=1$ and hence the pointed conic has weight $\le 
\gamma + bc_I + (-1-b)c_J + c_K$.  Let $I$ index all points off $x=0$ and $J$ all remaining points off $y=0$.  Semistability at the boundary values $b=-2$ and $b=-\frac{1}{2}$ translates into the inequalities $c-c_I \le \frac{2c-\gamma}{3}$ and $c - (c_I + c_J) \le c - 2(\frac{c+\gamma}{3})$, respectively, but these are satisfied by the assumption that there is $\le \frac{c-2\gamma}{3}$ weight on any component and $\le c - 2(\frac{c+\gamma}{3})$ weight at the node.

\emph{Case 2: $x=y=0$ is a smooth point}.  An easy computation with partial derivatives shows that for $x=y=0$ to be smooth the conic cannot consist only of the monomials $x^2,xy,y^2$, so there is a monomial with weight $\le \text{wt}(xz)=-1-b$ and hence the pointed conic has weight $\le \gamma(-1-b) + bc_I + (-1-b)c_J + c_K$.  The boundary values $b=-2$ and $b=-\frac{1}{2}$ now translate to $c-c_I \le \frac{2c-\gamma}{3}$ and $c-(c_I + c_J) \le \frac{c+\gamma}{3}$, respectively, which can be satisfied simultaneously by choosing $I,J$ as before and noting that there is $\le \frac{c+\gamma}{3}$ weight at the smooth point $x=y=0$.

\subsubsection*{Two conditions for (semi)stability of a nonsingular conic}

Consider the nonsingular curve $x^2 + xy + xz + y^2$.  For $b=-\frac{1}{2}$ the monomial of minimal weight is $xz$ so when computing $\mu$ we set $i=1,j=0$ and find, as before, that semistability implies $\text{min}\{c - (c_I + c_J)\} \le \frac{c+\gamma}{3}$ and hence the weight at $x=y=0$ is $\le \frac{c+\gamma}{3}$ so the weight at any point of a semistable nonsingular conic is $\le \frac{c+\gamma}{3}$.  For $b=-1$ the minimal weight monomial is either $xz$ or $y^2$ (both have weight $0$) so the total weight is $-c_I + c_K = c -2c_I - c_J$.  Semistability implies $\text{max}\{2c_I + c_J\} \ge c$.  This maximum occurs when $I$ indexes all points off the line $x=0$ and $J$ indexes all remaining points off $y=0$.  The line $x=0$ is tangent to our conic, intersecting it at the unique point $x=y=0$, so $I$ indexes all points away from $x=y=0$ and $J=\varnothing$.  This means $c_J=0$ so $2c_I \ge c$, or equivalently $c-c_I \le \frac{c}{2}$, and the number $c-c_I$ here measures the weight at $x=y=0$, so any point of a semistable nonsingular conic has weight $\le \frac{c}{2}$.  Together this implies semistable nonsingular conics have $\le \mathtt{min}\{\frac{c+\gamma}{3},\frac{c}{2}\}$ weight at any point.\\

Conversely, to see that any nonsingular conic satisfying these two inequalities is semistable takes a little more work.  First, we verify semistability for conics that are not tangent to the line $x=0$.  If the conic is in the $\mathbb{C}$-linear span of $\{x^2,xy,xz,y^2\}$ then the intersection with the line $x=0$ is the single point $x=y=0$ so $x=0$ is a tangent line.  Therefore, we can assume there is a monomial with weight $\le \text{wt}(yz)$ (irrespective of $b$) so the pointed conic has weight $\le \gamma b + bc_I + (1-b)c_J + c_K$.  For $b=-2$ this is $\le 0$ iff $c-c_I \le 2(\frac{c + \gamma}{3})$, so by taking $I$ to index all points off $x=0$ this inequality is satisfied since the line $x=0$ intersects the conic in exactly two points, each of which has $\le \frac{c+\gamma}{3}$ weight.  For $b=-\frac{1}{2}$ the weight is $\le -\frac{\gamma}{2} - \frac{c_I}{2} - \frac{c_J}{2} + c_K$ which is $\le 0$ iff $c-(c_I + c_J) \le \frac{c+\gamma}{3}$, so if $J$ indexes all remaining points off $x=0$ then the inequality is satisfied since the weight at the point $x=y=0$ is $\le \frac{c+\gamma}{3}$.

For conics tangent to the line $x=0$ we need to subdivide the interval $[-2,-\frac{1}{2}]$ 
to show that there is a single coordinate with weight $\le 0$ when $b=-2$ and $b=-1$, whence by linear interpolation for all $b\in [-2,-1]$, and then separately that there is a coordinate with weight $\le 0$ for $b=-1$ and $b=-\frac{1}{2}$.  A conic in the $\mathbb{C}$-linear span of $\{x^2,xy,xz\}$ contains the line $x=0$ and hence cannot be nonsingular, so there is a monomial of weight $\le\text{wt}(y^2)$ and thus the pointed conic has weight $\le \gamma(2b+2) + bc_I + (-1-b)c_J + c_K$.  For $b=-2$ this is $\le 0$ iff $c-c_I \le 2(\frac{c+\gamma}{3})$, and this inequality is easily satisfied when $I$ indexes the points off $x=0$ (in fact tangency gives the stronger condition that $c-c_I$ measures the weight at the single point $x=y=0$ so it is $\le\frac{c+\gamma}{3}$).  For $b=-1$ the weight is $\le -c_I +c_K=c-2c_I - c_J$ so setting $J=\varnothing$ it is enough to check that $c-c_I \le \frac{c}{2}$.  The condition that $x=0$ intersects the conic in a unique point forces this inequality to be satisfied.  Finally, to prove semistability for $b\in[-1,-\frac{1}{2}]$ we observe that all conics in the $\mathbb{C}$-linear span of $\{x^2,xy,y^2\}$ are singular so there must be a monomial of weight $\le \text{wt}(xz)$.  If $b=-1$ then $\text{wt}(xz)=\text{wt}(y^2)$ and we have already seen such conics have weight $\le 0$.  If $b=-\frac{1}{2}$ then to show the weight is $\le 0$ it is enough to show that $c-(c_I+c_J) \le \frac{c+\gamma}{3}$, which holds by assumption.

\subsubsection*{Conclusion}  By the numerical criterion, the inequalities characterizing semistability of a pointed conic give a characterization for stability if they are replaced by strict inequalities.  Therefore the proof of Theorem \ref{thm:stability} is complete except for the final remark that singular conics are unstable if $\gamma > \frac{c}{2}$.  But this is immediate since $\gamma > \frac{c}{2} \Rightarrow c-2(\frac{c+\gamma}{3})<0$ so a semistable nodal conic with $\gamma > \frac{c}{2}$ would have to satisfy the impossible condition of having negative weight at the node.

\section{Walls and chambers}

In this section we use Theorem \ref{thm:stability} to discuss and prove Corollary \ref{cor:chamber}.  The ample line bundles on $\mathbb{P}^5\times (\mathbb{P}^2)^n$ restrict to give ample line bundles on $\mathtt{Con}(n)$ and each comes with a unique linearization for the action of $\mathtt{SL}_3$ so $\mathtt{Con}(n)$ inherits the $\mathtt{SL}_3$-ample cone $\mathbb{Q}^{n+1}_{>0}$.  Because $\frac{c+\gamma}{3} \le \frac{c}{2} \Leftrightarrow \gamma \le \frac{c}{2}$ we see by looking at the weight allowed at a smooth point of a conic that for a linearization $(\gamma,c_1,\ldots,c_n)$ the semistable locus is empty when $\mathtt{max}\{c_i\} > \frac{c+\gamma}{3}$ for $\gamma \le \frac{c}{2}$ and when $\mathtt{max}\{c_i\} > \frac{c}{2}$ for $\gamma \ge \frac{c}{2}$.  This suggests taking cross-sections of the $\mathtt{SL}_3$-ample cone to normalize in the following way: \begin{equation} \begin{cases}c+\gamma = 3 \text{ if } \gamma \le \frac{c}{2} \cr c=2 \text{ if } \gamma \ge \frac{c}{2} \cr \end{cases} \label{eqn:normalization}\end{equation}  The intersection of these two hyperplanes in the space of linearizations is the locus where $c=2$ and $\gamma=1$, so it is useful to view $\gamma$ as a parameter in $[0,\infty)$ such that for $\gamma \ge 1$ the point weights satisfy $c=2$ and for $\gamma \le 1$ they satisfy $c = 3 - \gamma$.  

If $\gamma > 1$ then singular conics are all unstable so the only relevant stability conditions are those of a nonsingular conic, namely that the weight at any point is $\le 1$.  Therefore, for any such $\gamma$ the $c_i$ satisfy $0 < c_i \le 1$ and $\sum_{i=1}^n c_i = 2$ so the linearization polytope is $\Delta(2,n)$.  The walls are where strictly semistable points occur, namely $\sum_{i\in I} c_i = 1$ for $I\subset\{1,\ldots,n\}$.  This is the same polytope and chamber decomposition as the GIT quotient parametrizing $n$ points on $\mathbb{P}^1$.  In fact, because a nonsingular conic is isomorphic to $\mathbb{P}^1$ and the action on marked points of $\mathtt{SL}_3$ stabilizing such a conic is the same as that of $\mathtt{SL}_2$ on $\mathbb{P}^1$ with the same stability conditions there is an isomorphism $\mathtt{Con}(n)\quotient_{(\gamma,\vec{c})}\mathtt{SL}_3 \cong (\mathbb{P}^1)^n\quotient_{\vec{c}~}\mathtt{SL}_2$ for any $\gamma > 1$.

On the other hand, if $\gamma \le 1$ then the linearization polytope is $\Delta(3-\gamma,n)$ since $\mathtt{min}\{\frac{c+\gamma}{3},\frac{c}{2}\} = \frac{c+\gamma}{3} = 1$.  The stability conditions for points on a nonsingular conic, for the weight on each component of a nodal conic, and for the weight at a smooth point of a nodal conic all introduce walls of the form $\sum_{i\in I}c_i = 1$ for $I\subset\{1,\ldots,n\}$.  The only remaining condition, the weight at the node, introduces walls $\sum_{i\in I}c_i = 2$ because if a collection of points indexed by $I\subset\{1,\ldots,n\}$ has total weight $c - 2$ then the complementary subset $\{1,\ldots,n\}\setminus I$ has total weight $2$.  If we let $\gamma$ vary in the interval $[0,1]$ then the semistable locus is nonempty as long as all entries of the weight vector are between 0 and 1, so the linearization polytope for this range is $\Delta(3,n+1)$ with walls $\sum_{i\in I}c_i = 1, \sum_{i\in I}c_i = 2$ that are independent of $\gamma$.  This concludes the proof of Corollary \ref{cor:chamber}, but it is worthwhile discussing the polytopes that arise and how they fit together.

\begin{figure}
\begin{center}
\begin{picture}(135,70)(0,0)
\linethickness{0.3mm}
\multiput(45,35)(0,1.82){6}{\line(0,1){0.91}}
\linethickness{0.3mm}
\multiput(35,45)(1.82,0){6}{\line(1,0){0.91}}
\linethickness{0.3mm}
\multiput(40,30)(1.43,1.43){4}{\multiput(0,0)(0.12,0.12){6}{\line(1,0){0.12}}}
\linethickness{0.3mm}
\multiput(30,30)(1.82,0){6}{\line(1,0){0.91}}
\linethickness{0.3mm}
\multiput(30,30)(0,1.82){6}{\line(0,1){0.91}}
\linethickness{0.3mm}
\multiput(30,40)(1.43,1.43){4}{\multiput(0,0)(0.12,0.12){6}{\line(1,0){0.12}}}
\linethickness{0.3mm}
\multiput(40,40)(1.43,1.43){4}{\multiput(0,0)(0.12,0.12){6}{\line(1,0){0.12}}}
\linethickness{0.3mm}
\multiput(30,40)(1.82,0){6}{\line(1,0){0.91}}
\linethickness{0.3mm}
\multiput(40,30)(0,1.82){6}{\line(0,1){0.91}}
\linethickness{0.3mm}
\multiput(35,60)(1.54,-1.28){20}{\multiput(0,0)(0.15,-0.13){5}{\line(1,0){0.15}}}
\linethickness{0.3mm}
\multiput(20,20)(0.7,1.86){22}{\multiput(0,0)(0.12,0.31){3}{\line(0,1){0.31}}}
\put(40,40){\makebox(0,0)[cc]{*}}

\linethickness{0.3mm}
\put(35,35){\line(1,0){40}}
\linethickness{0.3mm}
\put(35,35){\line(0,1){35}}
\linethickness{0.3mm}
\multiput(10,10)(0.12,0.12){208}{\line(1,0){0.12}}
\linethickness{0.3mm}
\multiput(105,35)(0,1.82){6}{\line(0,1){0.91}}
\linethickness{0.3mm}
\multiput(95,45)(1.82,0){6}{\line(1,0){0.91}}
\linethickness{0.3mm}
\multiput(100,30)(1.43,1.43){4}{\multiput(0,0)(0.12,0.12){6}{\line(1,0){0.12}}}
\linethickness{0.3mm}
\multiput(90,30)(1.82,0){6}{\line(1,0){0.91}}
\linethickness{0.3mm}
\multiput(90,30)(0,1.82){6}{\line(0,1){0.91}}
\linethickness{0.3mm}
\multiput(90,40)(1.43,1.43){4}{\multiput(0,0)(0.12,0.12){6}{\line(1,0){0.12}}}
\linethickness{0.3mm}
\multiput(100,40)(1.43,1.43){4}{\multiput(0,0)(0.12,0.12){6}{\line(1,0){0.12}}}
\linethickness{0.3mm}
\multiput(90,40)(1.82,0){6}{\line(1,0){0.91}}
\linethickness{0.3mm}
\multiput(100,30)(0,1.82){6}{\line(0,1){0.91}}
\linethickness{0.3mm}
\put(95,35){\line(1,0){40}}
\linethickness{0.3mm}
\put(95,35){\line(0,1){35}}
\linethickness{0.3mm}
\multiput(70,10)(0.12,0.12){208}{\line(1,0){0.12}}
\linethickness{0.3mm}
\multiput(84.02,24.38)(1.89,0.61){18}{\multiput(0,0)(0.32,0.1){3}{\line(1,0){0.32}}}
\linethickness{0.3mm}
\multiput(84.02,24.2)(0.7,1.81){16}{\multiput(0,0)(0.12,0.3){3}{\line(0,1){0.3}}}
\linethickness{0.3mm}
\multiput(94.82,52.32)(1.57,-1.17){15}{\multiput(0,0)(0.16,-0.12){5}{\line(1,0){0.16}}}
\linethickness{0.3mm}
\multiput(90,39.82)(0.12,-0.12){81}{\line(1,0){0.12}}
\linethickness{0.3mm}
\multiput(100.09,30.09)(0.12,0.36){41}{\line(0,1){0.36}}
\linethickness{0.3mm}
\multiput(90,40.18)(0.38,0.12){39}{\line(1,0){0.38}}
\linethickness{0.3mm}
\multiput(20.36,20.36)(1.89,0.62){24}{\multiput(0,0)(0.32,0.1){3}{\line(1,0){0.32}}}
\end{picture}
\caption{Hypersimplices $\Delta(3,3)$ (left) and $\Delta(2,3)$ (right) corresponding to $n=3$, $\gamma = 0$ and $\gamma \ge 1$, respectively.}
\label{fig:delta3}
\end{center}
\end{figure}
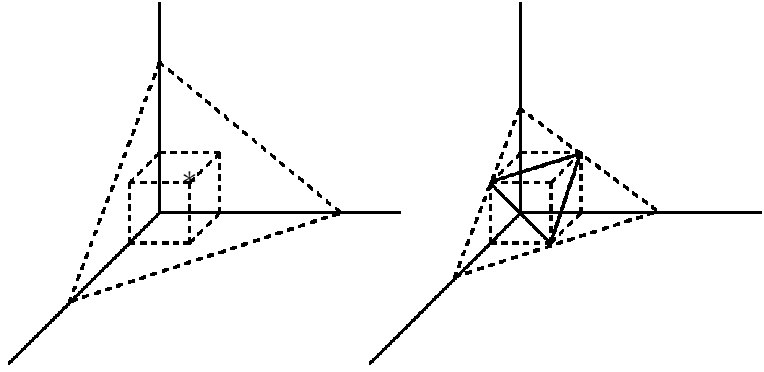

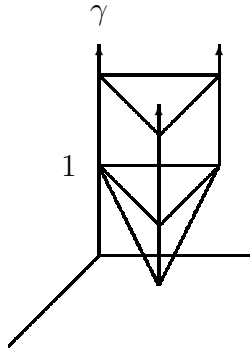
\begin{figure} 
\begin{center}
\begin{picture}(40,55)(0,0)
\linethickness{0.3mm}
\multiput(0,0)(0.12,0.12){125}{\line(1,0){0.12}}
\linethickness{0.3mm}
\put(15,15){\line(1,0){25}}
\linethickness{0.3mm}
\put(15,15){\line(0,1){35}}
\put(15,50){\vector(0,1){0.12}}
\put(15,55){\makebox(0,0)[cc]{$\gamma$}}
\linethickness{0.3mm}
\multiput(15,45)(0.12,-0.12){83}{\line(1,0){0.12}}
\linethickness{0.3mm}
\multiput(25,35)(0.12,0.12){83}{\line(1,0){0.12}}
\linethickness{0.3mm}
\put(15,45){\line(1,0){20}}
\linethickness{0.3mm}
\multiput(15,30)(0.12,-0.12){83}{\line(1,0){0.12}}
\linethickness{0.3mm}
\multiput(25,20)(0.12,0.12){83}{\line(1,0){0.12}}
\linethickness{0.3mm}
\put(15,30){\line(1,0){20}}
\linethickness{0.3mm}
\multiput(15,30)(0.12,-0.24){83}{\line(0,-1){0.24}}
\linethickness{0.3mm}
\multiput(25,10)(0.12,0.24){83}{\line(0,1){0.24}}
\linethickness{0.3mm}
\put(25,10){\line(0,1){10}}
\linethickness{0.3mm}
\put(25,20){\line(0,1){20}}
\put(25,40){\vector(0,1){0.12}}
\linethickness{0.3mm}
\put(35,30){\line(0,1){20}}
\put(35,50){\vector(0,1){0.12}}
\put(10,30){\makebox(0,0)[cc]{1}}
\end{picture}
\caption{The space of normalized linearizations for $\mathtt{Con}(3)$.}
\label{fig:linpoly3}
\end{center}
\end{figure}

The hypersimplex $\Delta(k,n)$ is sometimes described as the convex hull of all sums of $k$ distinct unit coordinate vectors in $\mathbb{R}^n$.  However, in our setting $k$ can be rational so the relevant definition is the intersection of the hypercube $[0,1]^n$ with a scaled copy of the standard coordinate simplex.  For $n=3$ we can draw the scaled 2-simplex and see how it intersects the cube $[0,1]^3$.  When $k=3$ (so $\gamma=0$) the intersection is a single point, namely $\Delta(3,3) = \{(1,1,1)\}$.  As $k$ varies from 3 to 2 (so $\gamma$ goes from 0 to 1) the 2-simplex shrinks so that the cube cuts out a larger piece of it.   The hypersimplex $\Delta(2,3)$ is a 2-simplex dual to the ambient 2-simplex.  See Figure \ref{fig:delta3}.  Putting this together gives a picture of the space of normalized linearizations for $n=3$ (admittedly an uninteresting case from a GIT/moduli perspective, but a helpful illustration nonetheless).  See Figure \ref{fig:linpoly3}.  The polytope obtained by restricting $\gamma$ to the interval $[0,1]$ is seen to be the tetrahedron which, as we discuss momentarily, is the hypersimplex $\Delta(3,4)$.  This exemplifies the general result that the polytope of linearizations for $\gamma$ on this interval is $\Delta(3,n+1)$.

\begin{figure}
\begin{center}
\begin{picture}(152.5,75.18)(0,0)
\linethickness{0.3mm}
\multiput(7.14,25.18)(1.75,-1.03){21}{\multiput(0,0)(0.22,-0.13){4}{\line(1,0){0.22}}}
\linethickness{0.3mm}
\multiput(43.04,4.11)(1.7,1.07){21}{\multiput(0,0)(0.21,0.13){4}{\line(1,0){0.21}}}
\linethickness{0.3mm}
\multiput(39.64,73.57)(1.25,-1.56){31}{\multiput(0,0)(0.13,-0.16){5}{\line(0,-1){0.16}}}
\linethickness{0.3mm}
\multiput(7.14,25.18)(1.1,1.64){30}{\multiput(0,0)(0.11,0.16){5}{\line(0,1){0.16}}}
\linethickness{0.3mm}
\multiput(39.82,73.21)(0.09,-2){35}{\multiput(0,0)(0.05,-1){1}{\line(0,-1){1}}}
\linethickness{0.3mm}
\multiput(7.32,25.71)(1.98,0.01){36}{\multiput(0,0)(0.99,0){1}{\line(1,0){0.99}}}
\linethickness{0.3mm}
\multiput(30,37.14)(0.12,-0.16){118}{\line(0,-1){0.16}}
\linethickness{0.3mm}
\multiput(43.93,17.5)(0.12,0.19){110}{\line(0,1){0.19}}
\linethickness{0.3mm}
\multiput(30.54,36.96)(4.49,0.12){6}{\line(1,0){4.49}}
\linethickness{0.3mm}
\multiput(30.54,37.5)(0.32,0.12){31}{\line(1,0){0.32}}
\linethickness{0.3mm}
\multiput(40.36,41.07)(0.63,-0.12){27}{\line(1,0){0.63}}
\linethickness{0.3mm}
\multiput(40.36,41.07)(0.12,-0.83){28}{\line(0,-1){0.83}}
\linethickness{0.3mm}
\multiput(99.05,25.99)(1.7,-1.06){15}{\multiput(0,0)(0.21,-0.13){4}{\line(1,0){0.21}}}
\linethickness{0.3mm}
\multiput(123.73,10.59)(1.65,1.11){15}{\multiput(0,0)(0.17,0.11){5}{\line(1,0){0.17}}}
\linethickness{0.3mm}
\multiput(121.4,61.35)(1.22,-1.61){22}{\multiput(0,0)(0.12,-0.16){5}{\line(0,-1){0.16}}}
\linethickness{0.3mm}
\multiput(99.05,25.99)(1.09,1.72){21}{\multiput(0,0)(0.11,0.17){5}{\line(0,1){0.17}}}
\linethickness{0.3mm}
\multiput(121.51,61.09)(0.09,-1.98){26}{\multiput(0,0)(0.04,-0.99){1}{\line(0,-1){0.99}}}
\linethickness{0.3mm}
\multiput(99.17,26.38)(1.97,0.01){25}{\multiput(0,0)(0.98,0){1}{\line(1,0){0.98}}}
\put(6.96,19.29){\makebox(0,0)[cc]{(3,0,0,0)}}
\put(46.25,75.18){\makebox(0,0)[cc]{(0,0,3,0)}}
\put(51.07,2.32){\makebox(0,0)[cc]{(0,3,0,0)}}
\put(80.18,20.71){\makebox(0,0)[cc]{(0,0,0,3)}}
\put(23.57,33.75){\makebox(0,0)[cc]{(1,1,1,0)}}
\put(34.29,44.11){\makebox(0,0)[cc]{(1,0,1,1)}}
\put(63.57,34.46){\makebox(0,0)[cc]{(0,1,1,1)}}
\put(49.46,15.36){\makebox(0,0)[cc]{(1,1,0,1)}}
\put(97.76,21.02){\makebox(0,0)[cc]{(2,0,0,0)}}
\put(127.32,63.75){\makebox(0,0)[cc]{(0,0,2,0)}}
\put(128.75,8.04){\makebox(0,0)[cc]{(0,2,0,0)}}
\put(152.5,22.68){\makebox(0,0)[cc]{(0,0,0,2)}}
\put(103.39,45){\makebox(0,0)[cc]{(1,0,1,0)}}
\put(142.86,44.11){\makebox(0,0)[cc]{(0,0,1,1)}}
\put(141.96,15.54){\makebox(0,0)[cc]{(0,1,0,1)}}
\put(108.93,14.11){\makebox(0,0)[cc]{(1,1,0,0)}}
\linethickness{0.3mm}
\multiput(111.25,18.39)(0.12,0.16){95}{\line(0,1){0.16}}
\linethickness{0.3mm}
\multiput(110.18,42.86)(0.17,-0.12){74}{\line(1,0){0.17}}
\linethickness{0.3mm}
\multiput(110.18,42.86)(0.12,-2.74){9}{\line(0,-1){2.74}}
\linethickness{0.3mm}
\multiput(122.86,33.93)(0.12,-0.13){115}{\line(0,-1){0.13}}
\linethickness{0.3mm}
\multiput(136.07,41.79)(0.12,-3.75){6}{\line(0,-1){3.75}}
\linethickness{0.3mm}
\multiput(122.68,33.75)(0.2,0.12){67}{\line(1,0){0.2}}
\linethickness{0.3mm}
\multiput(110.54,43.04)(0.12,-0.19){88}{\line(0,-1){0.19}}
\linethickness{0.3mm}
\multiput(121.25,26.07)(0.12,0.12){124}{\line(0,1){0.12}}
\linethickness{0.3mm}
\multiput(110.36,42.32)(3.67,-0.13){7}{\line(1,0){3.67}}
\linethickness{0.3mm}
\multiput(111.25,18.04)(8.45,0.12){3}{\line(1,0){8.45}}
\linethickness{0.3mm}
\multiput(121.43,26.07)(0.24,-0.12){63}{\line(1,0){0.24}}
\linethickness{0.3mm}
\multiput(111.43,18.04)(0.15,0.12){67}{\line(1,0){0.15}}
\end{picture}
\caption{Hypersimplices $\Delta(3,4)$ (left) and $\Delta(2,4)$ (right) corresponding to $n=4$, $\gamma = 0$ and $\gamma \ge 1$, respectively.}
\label{fig:delta4}
\end{center}
\end{figure}
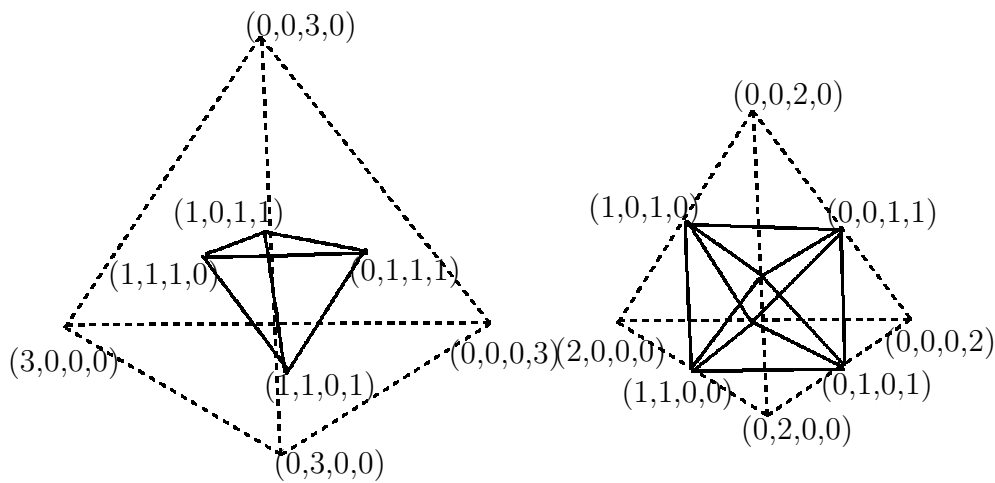

For $n=4$ the intersection of the scaled simplex and hypercube can be described by looking at what happens on each face of the simplex, since this just repeats the situation of $n=3$.  For $k=3$ on each face the only point of intersection is the center point of the face so $\Delta(3,4)$ is the convex hull of these four points, a 3-simplex dual to the ambient 3-simplex.  For $k=2$ the intersection with each face of the 3-simplex is a 2-simplex dual to that face and the convex hull of these forms an octahedron.  See Figure \ref{fig:delta4}.  For intermediate values $3 > k > 2$ the 3-simplex shrinks with respect to the hypercube so that $\Delta(k,4)$ is a truncated tetrahedron.  Thus, the linearization polytope $\Delta(3,5)$ for $n=4$ is obtained by starting with a tetrahedron and slowly truncating as $\gamma$ increase from 0 to 1 until it is truncated all the way to an octahedron.

To see why the union of the linearization polytopes $\Delta(3-\gamma,n)$ as $\gamma$ ranges in $[0,1]$ is $\Delta(3,n+1)$, note that slicing with the hyperplane $\gamma = 0$ yields the hypersimplex $\Delta(3,n)$ which is the convex hull of the $\binom{n}{3}$ points whose coordinates have a 0 in the first entry (for $\gamma$) and all other entries are 0 except for three entries with a 1.  Slicing with $\gamma = 1$ gives the hypersimplex $\Delta(2,n)$ whose $\binom{n}{2}$ coordinates have a 1 in the first entry (for $\gamma$) and the remaining entries are all zero except for two entries with a 1.  Collectively these vertices are the $\binom{n}{3}+\binom{n}{2} = \binom{n+1}{3}$ vertices of the hypersimplex $\Delta(3,n+1)$ whose coordinates are all vectors of length $n+1$ with three 1s and 0s elsewhere.

\section{Moduli stability and GIT stability}

In this section we prove Theorem \ref{thm:contraction}.  The first step is to normalize the GIT linearization vector $(\gamma,c_1,\ldots,c_n)$ so that it is compatible with the weight data in Hassett's moduli spaces.  Recall (see \cite{Hass}) that for $\mathcal{A} := \{a_1,\ldots,a_n\}\in\mathbb{Q}^n\cap [0,1]^n$ the moduli space $\overline{\mathcal{M}}_{0,\mathcal{A}}$ parametrizes $n$-pointed nodal rational curves for which the \emph{log canonical divisor} $K + a_1p_1 + \cdots + a_np_n$ (i.e. the dualizing sheaf twisted by the marked points $p_i$ with weights $a_i$) is ample and such that the $p_i$ are smooth and $\sum_{i\in I}a_i \le 1$ if $\{p_i\}_{i\in I}$ simultaneously collide.  In what follows we take the moduli weight vector $\mathcal{A}$ to be the GIT point weights $\vec{c}=(c_1,\ldots,c_n)$.  Because Hassett requires $c_i\in[0,1]$, the point weights in the GIT linearizations should also be in this range---and fortunately this is precisely the condition imposed by the normalization (\ref{eqn:normalization}) used to describe the polytope of linearizations.  Since Kapranov's work \cite{Kapr} addresses GIT quotients $(\mathbb{P}^1)\quotient\mathtt{SL}_2$, and our quotients $\mathtt{Con}(n)\quotient\mathtt{SL}_3$ are isomorphic to these for $\gamma > \frac{c}{2}$, where $c:=c_1+\cdots+c_n$, we can restrict attention to the part of the $\mathtt{SL}_3$-ample cone described by $\gamma \le \frac{c}{2}$.  By Corollary \ref{cor:chamber} this subcone admits a uniform cross-section $c+\gamma=3$ in such a way that rays are in bijection with the polytope $\Delta(3,n+1)$ and the point weights satisfy $c_i\in[0,1]$ as required.

\begin{remark} For fixed $\gamma$ the vector $(\gamma,\vec{c})\in\Delta(3,n+1)$ satisfies $\vec{c}\in\Delta(3-\gamma,n)$ and this forces all curves parametrized by $\overline{\mathcal{M}}_{0,\vec{c}}$ to be \emph{chains} of $\mathbb{P}^1$s.  Indeed, the dualizing sheaf restricted to a component with one node has degree $2g-2+1=-1$ so for a curve to be stable it must have $> 1 = \frac{\gamma+c}{3} \ge \frac{c}{3}$ weight on each such component so there can be at most two of them.
\end{remark}

There is a reduction morphism $\overline{\mathcal{M}}_{0,n} \rightarrow \overline{\mathcal{M}}_{0,\vec{c}}$ (\cite{Hass}, Theorem 4.1) which contracts a tree of $\mathbb{P}^1$s down to a chain of $\mathbb{P}^1$s.   

\begin{remark} To prove Theorem \ref{thm:contraction} it is enough to show that there is a morphism $\overline{\mathcal{M}}_{0,\vec{c}} \rightarrow \mathtt{Con}(n)\quotient_{(\gamma,\vec{c})}\mathtt{SL}_3$ when $(\gamma,\vec{c})$ lies in the interior of a GIT chamber.  Indeed, on the face $\gamma=1$ it is easy to see that $\mathtt{Con}(n)\quotient_{(\gamma,\vec{c})}\mathtt{SL}_3 \cong (\mathbb{P}^1)^n\quotient_{\vec{c}}~\mathtt{SL}_2$ (even though the universal curves are different), and Hassett's space $\overline{\mathcal{M}}_{0,\vec{c}}$ is by definition $(\mathbb{P}^1)^n\quotient_{\vec{c}}~\mathtt{SL}_2$ here since $c=2$, so we can assume $\gamma < 1$, or equivalently $c > 2$.  Thus for a linearization $(\gamma,\vec{c})$ on any wall except those of the form $c_i = 0$ we can find a nearby linearization $(\gamma+\epsilon,\vec{c}-\epsilon)$ lying in an open GIT chamber.  By general principles of variation of GIT (\cite{Thad}, Theorem 2.3) any GIT quotient for linearization on a wall or boundary receives a morphism from a quotient for linearization in the interior of an adjacent chamber, so we may deduce the desired factorization for a linearization on a wall or boundary from that of one in an open chamber by considering the composition \[\overline{\mathcal{M}}_{0,n} \rightarrow \overline{\mathcal{M}}_{0,\vec{c}} \rightarrow \overline{\mathcal{M}}_{0,\vec{c}-\epsilon} \rightarrow \mathtt{Con}(n)\quotient_{(\gamma+\epsilon,\vec{c}-\epsilon)}\mathtt{SL}_3 \rightarrow \mathtt{Con}(n)\quotient_{(\gamma,\vec{c})}\mathtt{SL}_3.\] 
All the morphisms in this composition are clearly birational.  For the remained of this section, therefore, we assume $(\gamma,\vec{c})\in\Delta(3,n+1)$ lies in an open GIT chamber.
 \end{remark}

\begin{definition}
Let $\pi : \mathcal{C}_{0,\vec{c}} \rightarrow\overline{\mathcal{M}}_{0,\vec{c}}$ be the universal curve with universal sections $p_1,\ldots,p_n$.  Recall (\cite{Hass}, Proposition 5.4) that $\mathcal{C}_{0,\vec{c}} = \mathcal{M}_{0,\vec{c}\cup\{\epsilon\}}$ for sufficiently small $\epsilon > 0$.
\end{definition}

Intuitively, because no points have weight 1 the curve represented by a point in the moduli space is traced out by throwing in an additional marked point with such small weight that it can freely pass by all the other points.  

\begin{proposition} \label{prop:linebundle} The line bundle $L := \omega_\pi^{-1}$ on $\mathcal{C}_{0,\vec{c}}$ induces a morphism which embeds each stable nonsingular curve as a nonsingular conic and contracts the inner components of each singular stable curve and embeds the resulting curve as a nodal conic (see Figure \ref{fig:contract}).  Moreover, the  conics obtained this way are all GIT stable with respect to the linearization $(\gamma,\vec{c})$ and every GIT stable conic comes from a moduli-stable curve in this manner.
\end{proposition}

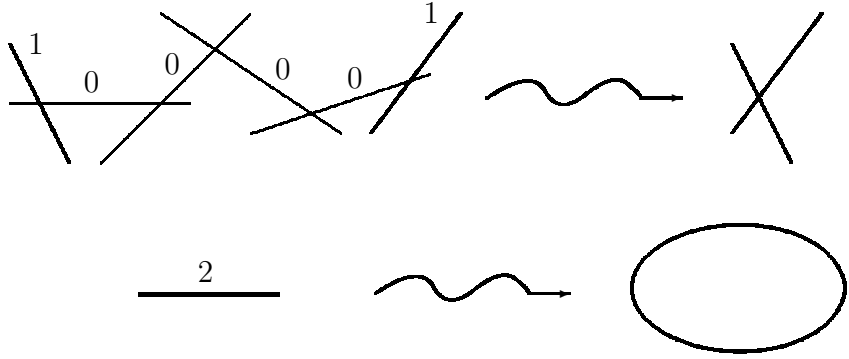
\begin{figure}
\begin{center}
\begin{picture}(140.89,57.5)(0,0)
\linethickness{0.4mm}
\multiput(2.14,52.5)(0.12,-0.24){83}{\line(0,-1){0.24}}
\linethickness{0.25mm}
\put(2.14,42.5){\line(1,0){30}}
\linethickness{0.25mm}
\multiput(17.14,32.5)(0.12,0.12){208}{\line(1,0){0.12}}
\linethickness{0.25mm}
\multiput(27.14,57.5)(0.18,-0.12){167}{\line(1,0){0.18}}
\linethickness{0.25mm}
\multiput(42.14,37.5)(0.36,0.12){83}{\line(1,0){0.36}}
\linethickness{0.4mm}
\multiput(62.14,37.5)(0.12,0.16){125}{\line(0,1){0.16}}
\linethickness{0.4mm}
\multiput(122.14,52.5)(0.12,-0.24){83}{\line(0,-1){0.24}}
\linethickness{0.4mm}
\multiput(122.14,37.5)(0.12,0.16){125}{\line(0,1){0.16}}
\put(6.43,52.5){\makebox(0,0)[cc]{1}}
\put(72.14,57.5){\makebox(0,0)[cc]{1}}
\put(15.71,46.07){\makebox(0,0)[cc]{0}}
\put(59.46,46.79){\makebox(0,0)[cc]{0}}
\put(29.11,49.11){\makebox(0,0)[cc]{0}}
\put(47.5,48.39){\makebox(0,0)[cc]{0}}
\linethickness{0.3mm}
\qbezier(81.43,43.39)(88.86,48.47)(91.07,44.51)
\qbezier(91.07,44.51)(93.28,40.55)(97.5,43.75)
\qbezier(97.5,43.75)(102.31,47.71)(104.64,45.85)
\qbezier(104.64,45.85)(106.98,43.98)(106.96,43.39)
\put(106.96,43.39){\line(1,0){6.79}}
\put(113.75,43.39){\vector(1,0){0.12}}
\linethickness{0.4mm}
\put(140.89,11.64){\line(0,1){0.3}}
\multiput(140.88,11.34)(0.01,0.3){1}{\line(0,1){0.3}}
\multiput(140.85,11.05)(0.03,0.3){1}{\line(0,1){0.3}}
\multiput(140.81,10.75)(0.04,0.29){1}{\line(0,1){0.29}}
\multiput(140.75,10.46)(0.06,0.29){1}{\line(0,1){0.29}}
\multiput(140.68,10.17)(0.07,0.29){1}{\line(0,1){0.29}}
\multiput(140.6,9.88)(0.08,0.29){1}{\line(0,1){0.29}}
\multiput(140.5,9.59)(0.1,0.29){1}{\line(0,1){0.29}}
\multiput(140.39,9.3)(0.11,0.29){1}{\line(0,1){0.29}}
\multiput(140.27,9.01)(0.12,0.29){1}{\line(0,1){0.29}}
\multiput(140.13,8.73)(0.14,0.28){1}{\line(0,1){0.28}}
\multiput(139.98,8.45)(0.15,0.28){1}{\line(0,1){0.28}}
\multiput(139.82,8.17)(0.16,0.28){1}{\line(0,1){0.28}}
\multiput(139.64,7.89)(0.18,0.28){1}{\line(0,1){0.28}}
\multiput(139.45,7.62)(0.09,0.14){2}{\line(0,1){0.14}}
\multiput(139.25,7.35)(0.1,0.13){2}{\line(0,1){0.13}}
\multiput(139.03,7.08)(0.11,0.13){2}{\line(0,1){0.13}}
\multiput(138.81,6.82)(0.11,0.13){2}{\line(0,1){0.13}}
\multiput(138.57,6.56)(0.12,0.13){2}{\line(0,1){0.13}}
\multiput(138.31,6.31)(0.13,0.13){2}{\line(0,1){0.13}}
\multiput(138.05,6.06)(0.13,0.13){2}{\line(1,0){0.13}}
\multiput(137.77,5.81)(0.14,0.12){2}{\line(1,0){0.14}}
\multiput(137.49,5.57)(0.14,0.12){2}{\line(1,0){0.14}}
\multiput(137.19,5.33)(0.15,0.12){2}{\line(1,0){0.15}}
\multiput(136.88,5.1)(0.15,0.12){2}{\line(1,0){0.15}}
\multiput(136.56,4.88)(0.16,0.11){2}{\line(1,0){0.16}}
\multiput(136.23,4.66)(0.17,0.11){2}{\line(1,0){0.17}}
\multiput(135.89,4.44)(0.17,0.11){2}{\line(1,0){0.17}}
\multiput(135.54,4.23)(0.18,0.1){2}{\line(1,0){0.18}}
\multiput(135.18,4.03)(0.18,0.1){2}{\line(1,0){0.18}}
\multiput(134.81,3.83)(0.18,0.1){2}{\line(1,0){0.18}}
\multiput(134.43,3.64)(0.19,0.1){2}{\line(1,0){0.19}}
\multiput(134.04,3.46)(0.19,0.09){2}{\line(1,0){0.19}}
\multiput(133.65,3.28)(0.4,0.18){1}{\line(1,0){0.4}}
\multiput(133.24,3.11)(0.4,0.17){1}{\line(1,0){0.4}}
\multiput(132.83,2.94)(0.41,0.16){1}{\line(1,0){0.41}}
\multiput(132.41,2.79)(0.42,0.16){1}{\line(1,0){0.42}}
\multiput(131.98,2.64)(0.43,0.15){1}{\line(1,0){0.43}}
\multiput(131.55,2.49)(0.43,0.14){1}{\line(1,0){0.43}}
\multiput(131.11,2.36)(0.44,0.14){1}{\line(1,0){0.44}}
\multiput(130.66,2.23)(0.45,0.13){1}{\line(1,0){0.45}}
\multiput(130.21,2.11)(0.45,0.12){1}{\line(1,0){0.45}}
\multiput(129.75,2)(0.46,0.11){1}{\line(1,0){0.46}}
\multiput(129.29,1.89)(0.46,0.11){1}{\line(1,0){0.46}}
\multiput(128.82,1.79)(0.47,0.1){1}{\line(1,0){0.47}}
\multiput(128.35,1.7)(0.47,0.09){1}{\line(1,0){0.47}}
\multiput(127.87,1.62)(0.48,0.08){1}{\line(1,0){0.48}}
\multiput(127.39,1.55)(0.48,0.07){1}{\line(1,0){0.48}}
\multiput(126.91,1.48)(0.48,0.07){1}{\line(1,0){0.48}}
\multiput(126.42,1.42)(0.49,0.06){1}{\line(1,0){0.49}}
\multiput(125.93,1.38)(0.49,0.05){1}{\line(1,0){0.49}}
\multiput(125.44,1.33)(0.49,0.04){1}{\line(1,0){0.49}}
\multiput(124.95,1.3)(0.49,0.03){1}{\line(1,0){0.49}}
\multiput(124.45,1.28)(0.49,0.02){1}{\line(1,0){0.49}}
\multiput(123.96,1.26)(0.5,0.02){1}{\line(1,0){0.5}}
\multiput(123.46,1.25)(0.5,0.01){1}{\line(1,0){0.5}}
\put(122.97,1.25){\line(1,0){0.5}}
\multiput(122.47,1.26)(0.5,-0.01){1}{\line(1,0){0.5}}
\multiput(121.98,1.28)(0.5,-0.02){1}{\line(1,0){0.5}}
\multiput(121.48,1.3)(0.49,-0.02){1}{\line(1,0){0.49}}
\multiput(120.99,1.33)(0.49,-0.03){1}{\line(1,0){0.49}}
\multiput(120.5,1.38)(0.49,-0.04){1}{\line(1,0){0.49}}
\multiput(120.01,1.42)(0.49,-0.05){1}{\line(1,0){0.49}}
\multiput(119.52,1.48)(0.49,-0.06){1}{\line(1,0){0.49}}
\multiput(119.04,1.55)(0.48,-0.07){1}{\line(1,0){0.48}}
\multiput(118.56,1.62)(0.48,-0.07){1}{\line(1,0){0.48}}
\multiput(118.08,1.7)(0.48,-0.08){1}{\line(1,0){0.48}}
\multiput(117.61,1.79)(0.47,-0.09){1}{\line(1,0){0.47}}
\multiput(117.14,1.89)(0.47,-0.1){1}{\line(1,0){0.47}}
\multiput(116.68,2)(0.46,-0.11){1}{\line(1,0){0.46}}
\multiput(116.22,2.11)(0.46,-0.11){1}{\line(1,0){0.46}}
\multiput(115.77,2.23)(0.45,-0.12){1}{\line(1,0){0.45}}
\multiput(115.32,2.36)(0.45,-0.13){1}{\line(1,0){0.45}}
\multiput(114.88,2.49)(0.44,-0.14){1}{\line(1,0){0.44}}
\multiput(114.45,2.64)(0.43,-0.14){1}{\line(1,0){0.43}}
\multiput(114.02,2.79)(0.43,-0.15){1}{\line(1,0){0.43}}
\multiput(113.6,2.94)(0.42,-0.16){1}{\line(1,0){0.42}}
\multiput(113.19,3.11)(0.41,-0.16){1}{\line(1,0){0.41}}
\multiput(112.78,3.28)(0.4,-0.17){1}{\line(1,0){0.4}}
\multiput(112.39,3.46)(0.4,-0.18){1}{\line(1,0){0.4}}
\multiput(112,3.64)(0.19,-0.09){2}{\line(1,0){0.19}}
\multiput(111.62,3.83)(0.19,-0.1){2}{\line(1,0){0.19}}
\multiput(111.25,4.03)(0.18,-0.1){2}{\line(1,0){0.18}}
\multiput(110.89,4.23)(0.18,-0.1){2}{\line(1,0){0.18}}
\multiput(110.54,4.44)(0.18,-0.1){2}{\line(1,0){0.18}}
\multiput(110.2,4.66)(0.17,-0.11){2}{\line(1,0){0.17}}
\multiput(109.87,4.88)(0.17,-0.11){2}{\line(1,0){0.17}}
\multiput(109.55,5.1)(0.16,-0.11){2}{\line(1,0){0.16}}
\multiput(109.24,5.33)(0.15,-0.12){2}{\line(1,0){0.15}}
\multiput(108.94,5.57)(0.15,-0.12){2}{\line(1,0){0.15}}
\multiput(108.65,5.81)(0.14,-0.12){2}{\line(1,0){0.14}}
\multiput(108.38,6.06)(0.14,-0.12){2}{\line(1,0){0.14}}
\multiput(108.11,6.31)(0.13,-0.13){2}{\line(1,0){0.13}}
\multiput(107.86,6.56)(0.13,-0.13){2}{\line(0,-1){0.13}}
\multiput(107.62,6.82)(0.12,-0.13){2}{\line(0,-1){0.13}}
\multiput(107.4,7.08)(0.11,-0.13){2}{\line(0,-1){0.13}}
\multiput(107.18,7.35)(0.11,-0.13){2}{\line(0,-1){0.13}}
\multiput(106.98,7.62)(0.1,-0.13){2}{\line(0,-1){0.13}}
\multiput(106.79,7.89)(0.09,-0.14){2}{\line(0,-1){0.14}}
\multiput(106.61,8.17)(0.18,-0.28){1}{\line(0,-1){0.28}}
\multiput(106.45,8.45)(0.16,-0.28){1}{\line(0,-1){0.28}}
\multiput(106.3,8.73)(0.15,-0.28){1}{\line(0,-1){0.28}}
\multiput(106.16,9.01)(0.14,-0.28){1}{\line(0,-1){0.28}}
\multiput(106.04,9.3)(0.12,-0.29){1}{\line(0,-1){0.29}}
\multiput(105.93,9.59)(0.11,-0.29){1}{\line(0,-1){0.29}}
\multiput(105.83,9.88)(0.1,-0.29){1}{\line(0,-1){0.29}}
\multiput(105.75,10.17)(0.08,-0.29){1}{\line(0,-1){0.29}}
\multiput(105.68,10.46)(0.07,-0.29){1}{\line(0,-1){0.29}}
\multiput(105.62,10.75)(0.06,-0.29){1}{\line(0,-1){0.29}}
\multiput(105.58,11.05)(0.04,-0.29){1}{\line(0,-1){0.29}}
\multiput(105.55,11.34)(0.03,-0.3){1}{\line(0,-1){0.3}}
\multiput(105.54,11.64)(0.01,-0.3){1}{\line(0,-1){0.3}}
\put(105.54,11.64){\line(0,1){0.3}}
\multiput(105.54,11.93)(0.01,0.3){1}{\line(0,1){0.3}}
\multiput(105.55,12.23)(0.03,0.3){1}{\line(0,1){0.3}}
\multiput(105.58,12.52)(0.04,0.29){1}{\line(0,1){0.29}}
\multiput(105.62,12.82)(0.06,0.29){1}{\line(0,1){0.29}}
\multiput(105.68,13.11)(0.07,0.29){1}{\line(0,1){0.29}}
\multiput(105.75,13.4)(0.08,0.29){1}{\line(0,1){0.29}}
\multiput(105.83,13.7)(0.1,0.29){1}{\line(0,1){0.29}}
\multiput(105.93,13.99)(0.11,0.29){1}{\line(0,1){0.29}}
\multiput(106.04,14.27)(0.12,0.29){1}{\line(0,1){0.29}}
\multiput(106.16,14.56)(0.14,0.28){1}{\line(0,1){0.28}}
\multiput(106.3,14.84)(0.15,0.28){1}{\line(0,1){0.28}}
\multiput(106.45,15.13)(0.16,0.28){1}{\line(0,1){0.28}}
\multiput(106.61,15.4)(0.18,0.28){1}{\line(0,1){0.28}}
\multiput(106.79,15.68)(0.09,0.14){2}{\line(0,1){0.14}}
\multiput(106.98,15.95)(0.1,0.13){2}{\line(0,1){0.13}}
\multiput(107.18,16.22)(0.11,0.13){2}{\line(0,1){0.13}}
\multiput(107.4,16.49)(0.11,0.13){2}{\line(0,1){0.13}}
\multiput(107.62,16.75)(0.12,0.13){2}{\line(0,1){0.13}}
\multiput(107.86,17.01)(0.13,0.13){2}{\line(0,1){0.13}}
\multiput(108.11,17.27)(0.13,0.13){2}{\line(1,0){0.13}}
\multiput(108.38,17.52)(0.14,0.12){2}{\line(1,0){0.14}}
\multiput(108.65,17.76)(0.14,0.12){2}{\line(1,0){0.14}}
\multiput(108.94,18)(0.15,0.12){2}{\line(1,0){0.15}}
\multiput(109.24,18.24)(0.15,0.12){2}{\line(1,0){0.15}}
\multiput(109.55,18.47)(0.16,0.11){2}{\line(1,0){0.16}}
\multiput(109.87,18.7)(0.17,0.11){2}{\line(1,0){0.17}}
\multiput(110.2,18.92)(0.17,0.11){2}{\line(1,0){0.17}}
\multiput(110.54,19.13)(0.18,0.1){2}{\line(1,0){0.18}}
\multiput(110.89,19.34)(0.18,0.1){2}{\line(1,0){0.18}}
\multiput(111.25,19.54)(0.18,0.1){2}{\line(1,0){0.18}}
\multiput(111.62,19.74)(0.19,0.1){2}{\line(1,0){0.19}}
\multiput(112,19.93)(0.19,0.09){2}{\line(1,0){0.19}}
\multiput(112.39,20.11)(0.4,0.18){1}{\line(1,0){0.4}}
\multiput(112.78,20.29)(0.4,0.17){1}{\line(1,0){0.4}}
\multiput(113.19,20.46)(0.41,0.16){1}{\line(1,0){0.41}}
\multiput(113.6,20.63)(0.42,0.16){1}{\line(1,0){0.42}}
\multiput(114.02,20.78)(0.43,0.15){1}{\line(1,0){0.43}}
\multiput(114.45,20.93)(0.43,0.14){1}{\line(1,0){0.43}}
\multiput(114.88,21.08)(0.44,0.14){1}{\line(1,0){0.44}}
\multiput(115.32,21.21)(0.45,0.13){1}{\line(1,0){0.45}}
\multiput(115.77,21.34)(0.45,0.12){1}{\line(1,0){0.45}}
\multiput(116.22,21.46)(0.46,0.11){1}{\line(1,0){0.46}}
\multiput(116.68,21.58)(0.46,0.11){1}{\line(1,0){0.46}}
\multiput(117.14,21.68)(0.47,0.1){1}{\line(1,0){0.47}}
\multiput(117.61,21.78)(0.47,0.09){1}{\line(1,0){0.47}}
\multiput(118.08,21.87)(0.48,0.08){1}{\line(1,0){0.48}}
\multiput(118.56,21.95)(0.48,0.07){1}{\line(1,0){0.48}}
\multiput(119.04,22.02)(0.48,0.07){1}{\line(1,0){0.48}}
\multiput(119.52,22.09)(0.49,0.06){1}{\line(1,0){0.49}}
\multiput(120.01,22.15)(0.49,0.05){1}{\line(1,0){0.49}}
\multiput(120.5,22.2)(0.49,0.04){1}{\line(1,0){0.49}}
\multiput(120.99,22.24)(0.49,0.03){1}{\line(1,0){0.49}}
\multiput(121.48,22.27)(0.49,0.02){1}{\line(1,0){0.49}}
\multiput(121.98,22.3)(0.5,0.02){1}{\line(1,0){0.5}}
\multiput(122.47,22.31)(0.5,0.01){1}{\line(1,0){0.5}}
\put(122.97,22.32){\line(1,0){0.5}}
\multiput(123.46,22.32)(0.5,-0.01){1}{\line(1,0){0.5}}
\multiput(123.96,22.31)(0.5,-0.02){1}{\line(1,0){0.5}}
\multiput(124.45,22.3)(0.49,-0.02){1}{\line(1,0){0.49}}
\multiput(124.95,22.27)(0.49,-0.03){1}{\line(1,0){0.49}}
\multiput(125.44,22.24)(0.49,-0.04){1}{\line(1,0){0.49}}
\multiput(125.93,22.2)(0.49,-0.05){1}{\line(1,0){0.49}}
\multiput(126.42,22.15)(0.49,-0.06){1}{\line(1,0){0.49}}
\multiput(126.91,22.09)(0.48,-0.07){1}{\line(1,0){0.48}}
\multiput(127.39,22.02)(0.48,-0.07){1}{\line(1,0){0.48}}
\multiput(127.87,21.95)(0.48,-0.08){1}{\line(1,0){0.48}}
\multiput(128.35,21.87)(0.47,-0.09){1}{\line(1,0){0.47}}
\multiput(128.82,21.78)(0.47,-0.1){1}{\line(1,0){0.47}}
\multiput(129.29,21.68)(0.46,-0.11){1}{\line(1,0){0.46}}
\multiput(129.75,21.58)(0.46,-0.11){1}{\line(1,0){0.46}}
\multiput(130.21,21.46)(0.45,-0.12){1}{\line(1,0){0.45}}
\multiput(130.66,21.34)(0.45,-0.13){1}{\line(1,0){0.45}}
\multiput(131.11,21.21)(0.44,-0.14){1}{\line(1,0){0.44}}
\multiput(131.55,21.08)(0.43,-0.14){1}{\line(1,0){0.43}}
\multiput(131.98,20.93)(0.43,-0.15){1}{\line(1,0){0.43}}
\multiput(132.41,20.78)(0.42,-0.16){1}{\line(1,0){0.42}}
\multiput(132.83,20.63)(0.41,-0.16){1}{\line(1,0){0.41}}
\multiput(133.24,20.46)(0.4,-0.17){1}{\line(1,0){0.4}}
\multiput(133.65,20.29)(0.4,-0.18){1}{\line(1,0){0.4}}
\multiput(134.04,20.11)(0.19,-0.09){2}{\line(1,0){0.19}}
\multiput(134.43,19.93)(0.19,-0.1){2}{\line(1,0){0.19}}
\multiput(134.81,19.74)(0.18,-0.1){2}{\line(1,0){0.18}}
\multiput(135.18,19.54)(0.18,-0.1){2}{\line(1,0){0.18}}
\multiput(135.54,19.34)(0.18,-0.1){2}{\line(1,0){0.18}}
\multiput(135.89,19.13)(0.17,-0.11){2}{\line(1,0){0.17}}
\multiput(136.23,18.92)(0.17,-0.11){2}{\line(1,0){0.17}}
\multiput(136.56,18.7)(0.16,-0.11){2}{\line(1,0){0.16}}
\multiput(136.88,18.47)(0.15,-0.12){2}{\line(1,0){0.15}}
\multiput(137.19,18.24)(0.15,-0.12){2}{\line(1,0){0.15}}
\multiput(137.49,18)(0.14,-0.12){2}{\line(1,0){0.14}}
\multiput(137.77,17.76)(0.14,-0.12){2}{\line(1,0){0.14}}
\multiput(138.05,17.52)(0.13,-0.13){2}{\line(1,0){0.13}}
\multiput(138.31,17.27)(0.13,-0.13){2}{\line(0,-1){0.13}}
\multiput(138.57,17.01)(0.12,-0.13){2}{\line(0,-1){0.13}}
\multiput(138.81,16.75)(0.11,-0.13){2}{\line(0,-1){0.13}}
\multiput(139.03,16.49)(0.11,-0.13){2}{\line(0,-1){0.13}}
\multiput(139.25,16.22)(0.1,-0.13){2}{\line(0,-1){0.13}}
\multiput(139.45,15.95)(0.09,-0.14){2}{\line(0,-1){0.14}}
\multiput(139.64,15.68)(0.18,-0.28){1}{\line(0,-1){0.28}}
\multiput(139.82,15.4)(0.16,-0.28){1}{\line(0,-1){0.28}}
\multiput(139.98,15.13)(0.15,-0.28){1}{\line(0,-1){0.28}}
\multiput(140.13,14.84)(0.14,-0.28){1}{\line(0,-1){0.28}}
\multiput(140.27,14.56)(0.12,-0.29){1}{\line(0,-1){0.29}}
\multiput(140.39,14.27)(0.11,-0.29){1}{\line(0,-1){0.29}}
\multiput(140.5,13.99)(0.1,-0.29){1}{\line(0,-1){0.29}}
\multiput(140.6,13.7)(0.08,-0.29){1}{\line(0,-1){0.29}}
\multiput(140.68,13.4)(0.07,-0.29){1}{\line(0,-1){0.29}}
\multiput(140.75,13.11)(0.06,-0.29){1}{\line(0,-1){0.29}}
\multiput(140.81,12.82)(0.04,-0.29){1}{\line(0,-1){0.29}}
\multiput(140.85,12.52)(0.03,-0.3){1}{\line(0,-1){0.3}}
\multiput(140.88,12.23)(0.01,-0.3){1}{\line(0,-1){0.3}}
\linethickness{0.4mm}
\put(23.57,10.71){\line(1,0){23.39}}
\linethickness{0.3mm}
\qbezier(62.86,10.89)(70.29,15.97)(72.5,12.01)
\qbezier(72.5,12.01)(74.71,8.05)(78.93,11.25)
\qbezier(78.93,11.25)(83.74,15.21)(86.07,13.35)
\qbezier(86.07,13.35)(88.4,11.48)(88.39,10.89)
\put(88.39,10.89){\line(1,0){6.79}}
\put(95.18,10.89){\vector(1,0){0.12}}
\put(34.64,14.46){\makebox(0,0)[cc]{2}}
\end{picture}
\caption{The morphisms induced by $L$ on fibers of $\pi : \mathcal{C}_{0,\vec{c}} \rightarrow\overline{\mathcal{M}}_{0,\vec{c}}$.}
\label{fig:contract}
\end{center}
\end{figure}

\begin{proof}
We first observe that on each fiber $L$ has vanishing higher cohomology.   Indeed, a fiber is a moduli-stable curve $C$ and $L|_C = \omega_C^{-1}$ so by Serre duality $h^1(L|_C) = h^0(\omega_C^{\otimes 2}) = 0$.  Now $L$ has degree 1 on extremal components and degree 0 on inner components of a singular curve and it has degree 2 on a nonsingular curve (as is indicated in Figure \ref{fig:contract}) so as long as $L$ is basepoint-free it induces a morphism which contracts inner components.  By the vanishing cohomology observation it is enough to check that $L$ is relatively basepoint-free, i.e. that $L|_C$ is basepoint-free for each fiber $C$.  But if $C\cong \mathbb{P}^1$ then $L|_C\cong\mathcal{O}_{\mathbb{P}^1}(2)$ so it is obvious and if $C$ is a chain of $\mathbb{P}^1$s then on each extremal component $C'\subset C$ we have $L|_C'\cong\mathcal{O}_{\mathbb{P}^1}(1)$ and on each inner component $C''\subset C$ we have $L|_{C''}\cong\mathcal{O}_{\mathbb{P}^1}$ so it is also clear: global sections are constant on the inner components and do not simultaneously vanish at any points of the extremal components.  Since $\text{deg}(L|_C)=2$ the image of each curve under the morphism induced by $L$ has degree 2, and it is mapped to $\mathbb{P}^2$ since $\text{dim }\Gamma(C,L|_C) = 3$.  Indeed, Riemann-Roch says $\chi(L|_C) = \text{deg}(L|_C) + 1 - p_a,$ but $p_a=0$ and $\chi(L|_C) = h^0(L|_C)$ since $h^1(L|_C)=0$.

We have shown that $L$ maps each moduli-stable curve to a degree 2 curve in $\mathbb{P}^2$, i.e. a conic.  To finish the proof it only remains to verify the claim about GIT stability.  Recall that we have normalized so that $\frac{c+\gamma}{3}=1$.  Thus, by Theorem \ref{thm:stability}, GIT stability for nodal conics is characterized by the following three conditions: 
\begin{enumerate}
\item there is $<  c - 2$ weight at the node 
\item there is $< 1$ weight at any smooth point, and 
\item there is $> 1$ weight on each component away from the node
\end{enumerate}
But these follow immediately from the fact that on the original moduli-stable curve 
\begin{enumerate}
\item the extremal components each have $> 1$ weight leaving $< c - 2$ weight remaining on the inner components---which are precisely the components that get contracted to the node of the conic (recall Figure \ref{fig:contract})
\item the smooth points of the conic come from smooth points on the extremal components of the origin curve, so there is $< 1$ weight at any such point (since we are assuming $\vec{c}$ lies in the interior of $[0,1]^n$), and 
\item there is $> 1$ weight on the extremal components, as we have already noted.
\end{enumerate}
A nonsingular rational curve gets embedded as a nonsingular GIT-stable conic because the only condition to check is that there is $<1$ at each point.  This shows that the image under $L$ of any stable rational curve is a stable conic, but it is easy to see that any stable conic can be obtained this way.
\end{proof}

We can use Proposition \ref{prop:linebundle} to complete the proof of Theorem \ref{thm:contraction}.  The line bundle $L$ maps $\mathcal{C}_{0,\vec{c}}$ with its universal sections to a flat family of pointed curves over $\overline{\mathcal{M}}_{0,\vec{c}}$ and it embeds this family in a $\mathbb{P}^2$-bundle.  This $\mathbb{P}^2$-bundle has an associated principal $\mathtt{PGL}_3$-bundle $P$.  Pulling back the $\mathbb{P}^2$-bundle along the structure morphism $P \rightarrow \overline{\mathcal{M}}_{0,\vec{c}}$ trivializes it so we get a flat family of pointed conics over $P$ embedded in the trivial $\mathbb{P}^2$-bundle over $P$.  Now $\mathbb{P}^5$ is the Hilbert scheme of conics, and $\mathtt{Con}(n)$ is the Hilbert scheme of $n$-pointed conics, so by the universal property of Hilbert schemes this family over $P$ induces a morphism $P \rightarrow \mathtt{Con}(n)$.  But this morphism factors through the stable locus $\mathtt{Con}(n)_{s}$ because by Proposition \ref{prop:linebundle} the image under $L$ of a moduli-stable curve is GIT-stable.  Since $\mathtt{Con}(n)\quotient_{(\gamma,\vec{c})}\mathtt{SL}_3$ is the categorical quotient of $\mathtt{Con}(n)_s$ by $\mathtt{SL}_3$ (recall that by assumption stability and semistability coincide) and $\mathtt{SL}_3$ acts through $\mathtt{PGL}_3$ via the canonical isogeny $\mathtt{SL}_3 \rightarrow \mathtt{PGL}_3$ we see that the composition \[P \rightarrow \mathtt{Con}(n)_{s} \rightarrow \mathtt{Con}(n)\quotient_{\vec{w}~}\mathtt{SL}_3\] is $\mathtt{PGL}_3$-invariant so it must factor through the categorical quotient of $P$ by $\mathtt{PGL}_3$, which by the definition of principal bundle is $\overline{\mathcal{M}}_{0,\vec{c}}$.  This means precisely that there is a morphism $\overline{\mathcal{M}}_{0,\vec{c}} \rightarrow \mathtt{Con}(n)\quotient_{\vec{w}~}\mathtt{SL}_3$, thus concluding the proof of Theorem \ref{thm:contraction}.

\section{Examples, further properties, and applications}

In this section we explore some properties and manifestations of the conic compactifications constructed in this paper.

\subsection{Semistable reduction}

The morphism described in Theorem \ref{thm:contraction} and its proof can be used to study semistable reduction in the spaces $\mathtt{Con}(n)\quotient\mathtt{SL}_3$.  Any 1-parameter family of semistable configurations of $n$ points on a conic must have a semistable limit since the GIT quotient is proper.  If the conics are nonsingular we can identify them with $\mathbb{P}^1$ and the limit may be computed by first finding the limit as a stable curve in $\overline{\mathcal{M}}_{0,n}$ and then looking at the image of this curve under the morphism $\overline{\mathcal{M}}_{0,n} \rightarrow \mathtt{Con}(n)\quotient\mathtt{SL}_3$.  Figure \ref{fig:limit} shows an example with $\gamma=\frac{1}{8},\vec{c}=(\frac{5}{8},\frac{5}{8},\frac{5}{8},\frac{5}{8},\frac{2}{8},\frac{1}{8})$.  

\begin{figure}
\begin{center}
\begin{picture}(101.43,66.61)(0,0)
\linethickness{0.3mm}
\put(21.07,53.74){\line(0,1){0.38}}
\multiput(21.04,53.36)(0.03,0.38){1}{\line(0,1){0.38}}
\multiput(20.99,52.98)(0.05,0.38){1}{\line(0,1){0.38}}
\multiput(20.91,52.61)(0.08,0.38){1}{\line(0,1){0.38}}
\multiput(20.81,52.23)(0.1,0.37){1}{\line(0,1){0.37}}
\multiput(20.68,51.87)(0.13,0.37){1}{\line(0,1){0.37}}
\multiput(20.53,51.5)(0.15,0.36){1}{\line(0,1){0.36}}
\multiput(20.36,51.15)(0.18,0.36){1}{\line(0,1){0.36}}
\multiput(20.16,50.8)(0.1,0.17){2}{\line(0,1){0.17}}
\multiput(19.93,50.46)(0.11,0.17){2}{\line(0,1){0.17}}
\multiput(19.69,50.13)(0.12,0.17){2}{\line(0,1){0.17}}
\multiput(19.42,49.8)(0.13,0.16){2}{\line(0,1){0.16}}
\multiput(19.13,49.49)(0.14,0.16){2}{\line(0,1){0.16}}
\multiput(18.82,49.19)(0.1,0.1){3}{\line(1,0){0.1}}
\multiput(18.49,48.9)(0.17,0.14){2}{\line(1,0){0.17}}
\multiput(18.14,48.63)(0.17,0.14){2}{\line(1,0){0.17}}
\multiput(17.77,48.36)(0.18,0.13){2}{\line(1,0){0.18}}
\multiput(17.39,48.12)(0.19,0.12){2}{\line(1,0){0.19}}
\multiput(16.99,47.88)(0.2,0.12){2}{\line(1,0){0.2}}
\multiput(16.57,47.67)(0.21,0.11){2}{\line(1,0){0.21}}
\multiput(16.14,47.47)(0.21,0.1){2}{\line(1,0){0.21}}
\multiput(15.7,47.28)(0.22,0.09){2}{\line(1,0){0.22}}
\multiput(15.25,47.11)(0.45,0.17){1}{\line(1,0){0.45}}
\multiput(14.78,46.97)(0.47,0.15){1}{\line(1,0){0.47}}
\multiput(14.3,46.83)(0.48,0.13){1}{\line(1,0){0.48}}
\multiput(13.82,46.72)(0.48,0.11){1}{\line(1,0){0.48}}
\multiput(13.33,46.62)(0.49,0.1){1}{\line(1,0){0.49}}
\multiput(12.83,46.55)(0.5,0.08){1}{\line(1,0){0.5}}
\multiput(12.33,46.49)(0.5,0.06){1}{\line(1,0){0.5}}
\multiput(11.83,46.45)(0.5,0.04){1}{\line(1,0){0.5}}
\multiput(11.32,46.43)(0.51,0.02){1}{\line(1,0){0.51}}
\put(10.82,46.43){\line(1,0){0.51}}
\multiput(10.31,46.45)(0.51,-0.02){1}{\line(1,0){0.51}}
\multiput(9.81,46.49)(0.5,-0.04){1}{\line(1,0){0.5}}
\multiput(9.31,46.55)(0.5,-0.06){1}{\line(1,0){0.5}}
\multiput(8.81,46.62)(0.5,-0.08){1}{\line(1,0){0.5}}
\multiput(8.32,46.72)(0.49,-0.1){1}{\line(1,0){0.49}}
\multiput(7.84,46.83)(0.48,-0.11){1}{\line(1,0){0.48}}
\multiput(7.36,46.97)(0.48,-0.13){1}{\line(1,0){0.48}}
\multiput(6.89,47.11)(0.47,-0.15){1}{\line(1,0){0.47}}
\multiput(6.44,47.28)(0.45,-0.17){1}{\line(1,0){0.45}}
\multiput(6,47.47)(0.22,-0.09){2}{\line(1,0){0.22}}
\multiput(5.57,47.67)(0.21,-0.1){2}{\line(1,0){0.21}}
\multiput(5.15,47.88)(0.21,-0.11){2}{\line(1,0){0.21}}
\multiput(4.75,48.12)(0.2,-0.12){2}{\line(1,0){0.2}}
\multiput(4.37,48.36)(0.19,-0.12){2}{\line(1,0){0.19}}
\multiput(4,48.63)(0.18,-0.13){2}{\line(1,0){0.18}}
\multiput(3.65,48.9)(0.17,-0.14){2}{\line(1,0){0.17}}
\multiput(3.32,49.19)(0.17,-0.14){2}{\line(1,0){0.17}}
\multiput(3.01,49.49)(0.1,-0.1){3}{\line(1,0){0.1}}
\multiput(2.72,49.8)(0.14,-0.16){2}{\line(0,-1){0.16}}
\multiput(2.45,50.13)(0.13,-0.16){2}{\line(0,-1){0.16}}
\multiput(2.21,50.46)(0.12,-0.17){2}{\line(0,-1){0.17}}
\multiput(1.98,50.8)(0.11,-0.17){2}{\line(0,-1){0.17}}
\multiput(1.78,51.15)(0.1,-0.17){2}{\line(0,-1){0.17}}
\multiput(1.61,51.5)(0.18,-0.36){1}{\line(0,-1){0.36}}
\multiput(1.46,51.87)(0.15,-0.36){1}{\line(0,-1){0.36}}
\multiput(1.33,52.23)(0.13,-0.37){1}{\line(0,-1){0.37}}
\multiput(1.23,52.61)(0.1,-0.37){1}{\line(0,-1){0.37}}
\multiput(1.15,52.98)(0.08,-0.38){1}{\line(0,-1){0.38}}
\multiput(1.1,53.36)(0.05,-0.38){1}{\line(0,-1){0.38}}
\multiput(1.07,53.74)(0.03,-0.38){1}{\line(0,-1){0.38}}
\put(1.07,53.74){\line(0,1){0.38}}
\multiput(1.07,54.12)(0.03,0.38){1}{\line(0,1){0.38}}
\multiput(1.1,54.5)(0.05,0.38){1}{\line(0,1){0.38}}
\multiput(1.15,54.88)(0.08,0.38){1}{\line(0,1){0.38}}
\multiput(1.23,55.25)(0.1,0.37){1}{\line(0,1){0.37}}
\multiput(1.33,55.63)(0.13,0.37){1}{\line(0,1){0.37}}
\multiput(1.46,55.99)(0.15,0.36){1}{\line(0,1){0.36}}
\multiput(1.61,56.36)(0.18,0.36){1}{\line(0,1){0.36}}
\multiput(1.78,56.71)(0.1,0.17){2}{\line(0,1){0.17}}
\multiput(1.98,57.06)(0.11,0.17){2}{\line(0,1){0.17}}
\multiput(2.21,57.4)(0.12,0.17){2}{\line(0,1){0.17}}
\multiput(2.45,57.73)(0.13,0.16){2}{\line(0,1){0.16}}
\multiput(2.72,58.06)(0.14,0.16){2}{\line(0,1){0.16}}
\multiput(3.01,58.37)(0.1,0.1){3}{\line(1,0){0.1}}
\multiput(3.32,58.67)(0.17,0.14){2}{\line(1,0){0.17}}
\multiput(3.65,58.96)(0.17,0.14){2}{\line(1,0){0.17}}
\multiput(4,59.23)(0.18,0.13){2}{\line(1,0){0.18}}
\multiput(4.37,59.5)(0.19,0.12){2}{\line(1,0){0.19}}
\multiput(4.75,59.74)(0.2,0.12){2}{\line(1,0){0.2}}
\multiput(5.15,59.98)(0.21,0.11){2}{\line(1,0){0.21}}
\multiput(5.57,60.19)(0.21,0.1){2}{\line(1,0){0.21}}
\multiput(6,60.39)(0.22,0.09){2}{\line(1,0){0.22}}
\multiput(6.44,60.58)(0.45,0.17){1}{\line(1,0){0.45}}
\multiput(6.89,60.75)(0.47,0.15){1}{\line(1,0){0.47}}
\multiput(7.36,60.89)(0.48,0.13){1}{\line(1,0){0.48}}
\multiput(7.84,61.03)(0.48,0.11){1}{\line(1,0){0.48}}
\multiput(8.32,61.14)(0.49,0.1){1}{\line(1,0){0.49}}
\multiput(8.81,61.24)(0.5,0.08){1}{\line(1,0){0.5}}
\multiput(9.31,61.31)(0.5,0.06){1}{\line(1,0){0.5}}
\multiput(9.81,61.37)(0.5,0.04){1}{\line(1,0){0.5}}
\multiput(10.31,61.41)(0.51,0.02){1}{\line(1,0){0.51}}
\put(10.82,61.43){\line(1,0){0.51}}
\multiput(11.32,61.43)(0.51,-0.02){1}{\line(1,0){0.51}}
\multiput(11.83,61.41)(0.5,-0.04){1}{\line(1,0){0.5}}
\multiput(12.33,61.37)(0.5,-0.06){1}{\line(1,0){0.5}}
\multiput(12.83,61.31)(0.5,-0.08){1}{\line(1,0){0.5}}
\multiput(13.33,61.24)(0.49,-0.1){1}{\line(1,0){0.49}}
\multiput(13.82,61.14)(0.48,-0.11){1}{\line(1,0){0.48}}
\multiput(14.3,61.03)(0.48,-0.13){1}{\line(1,0){0.48}}
\multiput(14.78,60.89)(0.47,-0.15){1}{\line(1,0){0.47}}
\multiput(15.25,60.75)(0.45,-0.17){1}{\line(1,0){0.45}}
\multiput(15.7,60.58)(0.22,-0.09){2}{\line(1,0){0.22}}
\multiput(16.14,60.39)(0.21,-0.1){2}{\line(1,0){0.21}}
\multiput(16.57,60.19)(0.21,-0.11){2}{\line(1,0){0.21}}
\multiput(16.99,59.98)(0.2,-0.12){2}{\line(1,0){0.2}}
\multiput(17.39,59.74)(0.19,-0.12){2}{\line(1,0){0.19}}
\multiput(17.77,59.5)(0.18,-0.13){2}{\line(1,0){0.18}}
\multiput(18.14,59.23)(0.17,-0.14){2}{\line(1,0){0.17}}
\multiput(18.49,58.96)(0.17,-0.14){2}{\line(1,0){0.17}}
\multiput(18.82,58.67)(0.1,-0.1){3}{\line(1,0){0.1}}
\multiput(19.13,58.37)(0.14,-0.16){2}{\line(0,-1){0.16}}
\multiput(19.42,58.06)(0.13,-0.16){2}{\line(0,-1){0.16}}
\multiput(19.69,57.73)(0.12,-0.17){2}{\line(0,-1){0.17}}
\multiput(19.93,57.4)(0.11,-0.17){2}{\line(0,-1){0.17}}
\multiput(20.16,57.06)(0.1,-0.17){2}{\line(0,-1){0.17}}
\multiput(20.36,56.71)(0.18,-0.36){1}{\line(0,-1){0.36}}
\multiput(20.53,56.36)(0.15,-0.36){1}{\line(0,-1){0.36}}
\multiput(20.68,55.99)(0.13,-0.37){1}{\line(0,-1){0.37}}
\multiput(20.81,55.63)(0.1,-0.37){1}{\line(0,-1){0.37}}
\multiput(20.91,55.25)(0.08,-0.38){1}{\line(0,-1){0.38}}
\multiput(20.99,54.88)(0.05,-0.38){1}{\line(0,-1){0.38}}
\multiput(21.04,54.5)(0.03,-0.38){1}{\line(0,-1){0.38}}
\linethickness{0.3mm}
\put(30.36,53.39){\line(1,0){25}}
\linethickness{0.3mm}
\put(78.75,42.68){\line(0,1){20}}
\linethickness{0.3mm}
\put(71.43,48.21){\line(1,0){30}}
\linethickness{0.3mm}
\put(93.39,42.86){\line(0,1){20}}
\linethickness{0.3mm}
\multiput(26.25,25.89)(0.12,-0.12){208}{\line(1,0){0.12}}
\linethickness{0.3mm}
\multiput(36.25,0.89)(0.12,0.12){208}{\line(1,0){0.12}}
\put(25.36,54.29){\makebox(0,0)[cc]{$\cong$}}
\linethickness{0.3mm}
\qbezier(60.54,52.5)(62.56,54.64)(63.71,52.52)
\qbezier(63.71,52.52)(64.86,50.41)(66.61,52.14)
\qbezier(66.61,52.14)(68.6,54.35)(70.31,54.15)
\qbezier(70.31,54.15)(72.03,53.95)(72.14,53.75)
\linethickness{0.3mm}
\qbezier(15.48,40.03)(18.42,39.91)(17.58,37.66)
\qbezier(17.58,37.66)(16.73,35.41)(19.19,35.22)
\qbezier(19.19,35.22)(22.16,35.17)(23.14,33.74)
\qbezier(23.14,33.74)(24.12,32.31)(24.04,32.09)
\linethickness{0.3mm}
\multiput(64.64,31.07)(0.12,0.13){58}{\line(0,1){0.13}}
\put(64.64,31.07){\vector(-1,-1){0.12}}
\linethickness{0.3mm}
\multiput(70.54,40)(0.12,-0.12){19}{\line(0,-1){0.12}}
\linethickness{0.3mm}
\multiput(71.61,55)(0.12,-0.16){10}{\line(0,-1){0.16}}
\linethickness{0.3mm}
\multiput(71.43,52.5)(0.12,0.12){9}{\line(0,1){0.12}}
\linethickness{0.3mm}
\multiput(24.46,31.96)(0.12,0.54){3}{\line(0,1){0.54}}
\linethickness{0.3mm}
\put(22.5,31.96){\line(1,0){1.96}}
\put(5.18,59.64){\makebox(0,0)[cc]{*}}
\put(10.18,61.43){\makebox(0,0)[cc]{*}}
\put(3.21,49.29){\makebox(0,0)[cc]{*}}
\put(17.86,48.21){\makebox(0,0)[cc]{*}}
\put(7.68,46.96){\makebox(0,0)[cc]{*}}
\put(20.36,51.25){\makebox(0,0)[cc]{*}}
\put(33.39,53.21){\makebox(0,0)[cc]{*}}
\put(37.14,53.21){\makebox(0,0)[cc]{*}}
\put(41.79,53.21){\makebox(0,0)[cc]{*}}
\put(45.36,53.04){\makebox(0,0)[cc]{*}}
\put(50.36,53.04){\makebox(0,0)[cc]{*}}
\put(54.11,53.21){\makebox(0,0)[cc]{*}}
\put(78.39,57.14){\makebox(0,0)[cc]{*}}
\put(2.68,63.75){\makebox(0,0)[cc]{$\frac{5}{8}$}}
\put(10.18,66.61){\makebox(0,0)[cc]{$\frac{5}{8}$}}
\put(1.43,46.07){\makebox(0,0)[cc]{$\frac{1}{8}$}}
\put(6.43,43.75){\makebox(0,0)[cc]{$\frac{2}{8}$}}
\put(18.57,44.83){\makebox(0,0)[cc]{$\frac{5}{8}$}}
\put(22.5,48.39){\makebox(0,0)[cc]{$\frac{5}{8}$}}
\put(32.5,58.39){\makebox(0,0)[cc]{$\frac{5}{8}$}}
\put(37.5,58.21){\makebox(0,0)[cc]{$\frac{5}{8}$}}
\put(41.96,58.39){\makebox(0,0)[cc]{$\frac{1}{8}$}}
\put(45.71,58.21){\makebox(0,0)[cc]{$\frac{2}{8}$}}
\put(54.29,58.03){\makebox(0,0)[cc]{$\frac{5}{8}$}}
\put(50.36,58.22){\makebox(0,0)[cc]{$\frac{5}{8}$}}
\put(74.82,58.57){\makebox(0,0)[cc]{$\frac{5}{8}$}}
\put(74.82,52.5){\makebox(0,0)[cc]{$\frac{5}{8}$}}
\put(48.39,13.21){\makebox(0,0)[cc]{*}}
\put(78.57,52.68){\makebox(0,0)[cc]{*}}
\put(82.5,48.04){\makebox(0,0)[cc]{*}}
\put(88.21,47.86){\makebox(0,0)[cc]{*}}
\put(36.43,15.36){\makebox(0,0)[cc]{*}}
\put(93.21,57.5){\makebox(0,0)[cc]{*}}
\put(93.21,52.32){\makebox(0,0)[cc]{*}}
\put(32.5,19.64){\makebox(0,0)[cc]{*}}
\put(43.75,8.57){\makebox(0,0)[cc]{*}}
\put(53.04,17.68){\makebox(0,0)[cc]{*}}
\put(82.68,44.46){\makebox(0,0)[cc]{$\frac{1}{8}$}}
\put(88.57,44.28){\makebox(0,0)[cc]{$\frac{2}{8}$}}
\put(29.82,17.14){\makebox(0,0)[cc]{$\frac{5}{8}$}}
\put(97.14,52.14){\makebox(0,0)[cc]{$\frac{5}{8}$}}
\put(97.14,58.57){\makebox(0,0)[cc]{$\frac{5}{8}$}}
\put(55,15.36){\makebox(0,0)[cc]{$\frac{5}{8}$}}
\put(51.07,11.25){\makebox(0,0)[cc]{$\frac{5}{8}$}}
\put(34.11,13.75){\makebox(0,0)[cc]{$\frac{5}{8}$}}
\put(43.75,14.64){\makebox(0,0)[cc]{$\frac{3}{8}$}}
\linethickness{0.3mm}
\qbezier(7.16,58.24)(9.4,57.58)(10.25,58.43)
\qbezier(10.25,58.43)(11.1,59.29)(11.06,59.47)
\linethickness{0.3mm}
\multiput(7.25,58.27)(0.26,0.12){6}{\line(1,0){0.26}}
\linethickness{0.3mm}
\multiput(7.36,58.03)(0.12,-0.25){5}{\line(0,-1){0.25}}
\linethickness{0.3mm}
\qbezier(18.9,52.86)(16.6,52.5)(16.2,51.36)
\qbezier(16.2,51.36)(15.8,50.22)(15.92,50.07)
\linethickness{0.3mm}
\multiput(16.02,50)(0.27,0.13){5}{\line(1,0){0.27}}
\linethickness{0.3mm}
\multiput(14.86,50.77)(0.16,-0.12){7}{\line(1,0){0.16}}
\linethickness{0.3mm}
\qbezier(32.85,50.66)(34.82,49.42)(35.87,50.01)
\qbezier(35.87,50.01)(36.92,50.61)(36.93,50.79)
\linethickness{0.3mm}
\multiput(35.48,51.31)(0.47,-0.13){3}{\line(1,0){0.47}}
\linethickness{0.3mm}
\multiput(37.02,50.95)(0.12,-0.69){2}{\line(0,-1){0.69}}
\linethickness{0.3mm}
\qbezier(50.61,50.39)(52.59,49.15)(53.64,49.75)
\qbezier(53.64,49.75)(54.69,50.34)(54.7,50.53)
\linethickness{0.3mm}
\multiput(50.57,50.6)(1.48,0.11){1}{\line(1,0){1.48}}
\linethickness{0.3mm}
\multiput(50.5,50.24)(0.12,-0.69){2}{\line(0,-1){0.69}}
\end{picture}
\caption{An example of semistable reduction using the morphism $\overline{\mathcal{M}}_{0,n} \rightarrow \mathtt{Con}(n)\quotient\mathtt{SL}_3$ in Theorem \ref{thm:contraction}.}
\label{fig:limit}
\end{center}
\end{figure}
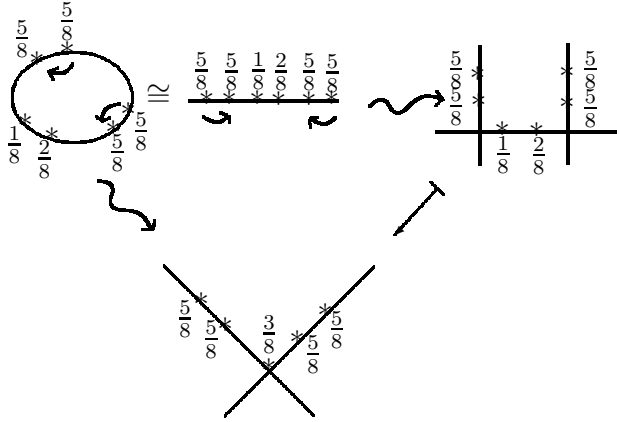

\subsection{Kontsevich-Boggi compactifications}

In \cite{Kont}, Kontsevich described certain topological modifications of the moduli spaces $\overline{\mathcal{M}}_{g,n}$ which for $g=0$ were given an algebraic and moduli-theoretic description by Boggi in \cite{Bogg}.  In the latter paper a rational pointed curve $(C,p_1,\ldots,p_n)$ is called \emph{$I$-stable}, where $I\subset \{1,\ldots,n\}$ is nonempty, if i) $C$ has at worst ordinary multiple points and the points $p_i$ with $i\in I$ avoid the singularities, ii) each component of $C$ has at least 3 special points (marked points or singularities) and at least 1 point indexed by $I$, and iii) any collection of points is allowed to collide as long as none of the points are indexed by $I$.  Boggi proves there is a normal projective variety, call it $\overline{\mathcal{M}}_{0,I}$, representing over $\text{Spec }\mathbb{Z}$ the moduli-functor of $I$-stable rational curves.  Moreover, he shows there is a sequence of birational morphisms \[\overline{\mathcal{M}}_{0,n} \rightarrow \overline{\mathcal{M}}_{0,[n]} \rightarrow \overline{\mathcal{M}}_{0,[n-1]} \rightarrow \cdots \rightarrow \overline{\mathcal{M}}_{0,[2]} \rightarrow \overline{\mathcal{M}}_{0,[1]} = \mathbb{P}^{n-3}\]
where $[j] := \{1,\ldots,j\}$.  Note that each stable curve parametrized by $\overline{\mathcal{M}}_{0,[j]}$ has at most $j$ components since each component must have a point indexed by $[j]$ and such points cannot lie on two components simultaneously.

The smallest Boggi space, $\overline{\mathcal{M}}_{0,[1]}$, can be constructed as an $\mathtt{SL}_2$ GIT quotient.  Indeed, this space parametrizes configurations of $n$ points on $\mathbb{P}^1$ such that $p_1$ is distinct from the others and the remaining $n-1$ points are supported in at least 2 points to ensure there are at least 3 special points.  Therefore, by taking a linearization such as $L=(1-\epsilon,\frac{1+\epsilon}{n-1},\ldots,\frac{1+\epsilon}{n-1})$ for small $\epsilon >0$ we get $(\mathbb{P}^1)^n\quotient_L\mathtt{SL}_2 \cong \overline{\mathcal{M}}_{0,[1]}$.  The next Boggi compactification, $\overline{\mathcal{M}}_{0,[2]}$, arises as an $\mathtt{SL}_3$ GIT quotient, namely \[\mathtt{Con}(n)\quotient_{(\gamma,\vec{c})}\mathtt{SL}_3 \cong \overline{\mathcal{M}}_{0,[2]}, \text{ where } (\gamma,\vec{c})=(3\epsilon,1-\epsilon,1-\epsilon,\frac{1-\epsilon}{n-2},\ldots,\frac{1-\epsilon}{n-2}).\]  Indeed, a nonsingular $[2]$-stable curve is a configuration of $n$ points on $\mathbb{P}^1$ such that $p_1$ and $p_2$ do not overlap with any points but the remaining $n-2$ can overlap---and a singular $[2]$-stable curve is a nodal curve with two components each marked by one of the heavy points $p_i$ ($i=1,2$) distinct from the rest as well as at least one other point away from the node.  This agrees with the GIT stability conditions prescribed by such a linearization.
 
\subsection{Contraction of $F$-curves}

The boundary of $\overline{\mathcal{M}}_{0,n}$ is stratified by topological type, and irreducible components of the 1-strata are called $F$-curves (or sometimes \emph{vital} curves).  The curves parametrized by an $F$-curve have one \emph{spine} which is a component with 4 special points, and up to 4 \emph{legs} which are chains of $\mathbb{P}^1$s attached to the spine at these points.  The cross-ratio of the 4 points on the spine traces a $\mathbb{P}^1$ in $\overline{\mathcal{M}}_{0,n}$.  $F$-curves are determined up to linear equivalence by a partition of $\{1,\ldots,n\}$ into 4 nonempty subsets---corresponding to the indices of the points on each leg.  The \emph{$F$-conjecture} is the statement that $F$-curves span the entire cone of curves, i.e. that every curve is linearly equivalent to a non-negative sum of $F$-curves.  This is known for $n\le 7$ in general and $n \le 24$ when considering an $S_n$-invariant analogue (see, e.g., \cite{GKM}).  

Given an $F$-curve corresponding to $\{1,\ldots,n\}=N_1\cup N_2 \cup N_3 \cup N_4$ and a weight vector $\vec{c} = (c_1,\ldots,c_n)$ we write $x_k = \sum_{i\in N_k}c_i$, so that $x_k$ measures the total weight on the $k^{\text{th}}$ leg.  We can assume without loss of generality $x_1 \le x_2 \le x_3 \le x_4$.  It is easy to see that the morphism $\overline{\mathcal{M}}_{0,n} \rightarrow \overline{\mathcal{M}}_{0,\vec{c}}$ contracts precisely those $F$-curves which satisfy $x_1 + x_2 + x_3 \le 1$ since this condition is necessary and sufficient for the spine to get contracted.  It is also not hard to see what remaining $F$-curves get contracted by the morphism $\overline{\mathcal{M}}_{0,\vec{c}} \rightarrow \mathtt{Con}(n)\quotient_{(\gamma,\vec{c})}\mathtt{SL}_3$, where $\gamma = 3 - \sum c_i$.

\begin{lemma}\label{F-curves}
The $F$-curves contracted by the morphism $\overline{\mathcal{M}}_{0,n} \rightarrow \mathtt{Con}(n)\quotient_{(\gamma,\vec{c})}\mathtt{SL}_3$ for $0 < \gamma < 1$ are precisely those satisfying $x_1 + x_2 + x_3 \le 1$ or $x_3 \ge 1$.
\end{lemma}

\begin{proof} If $x_3 > 1$ then the two heaviest legs become the components of a nodal conic and the spine is contracted to the node (and if $x_3 = 1$ then GIT semistable equivalence identifies this with the preceding case, hence also contracts the spine).  Conversely, if $x_3 \le 1$ but $x_1 + x_2 + x_3 > 1$ then the three lightest legs are each contracted to a point in Hassett's space and then the spine remains un-contracted as one of possibly two components in the resulting conic.
\end{proof}

\subsection{Inverse limits}

In \cite{Kapr} it is shown, using the general machinery of Chow quotients for torus actions and the Gelfand-MacPherson isomorphism $(\mathbb{P}^1)^n\quotient\mathtt{SL}_2 \cong \mathtt{Gr}(2,n)\quotient(\mathbb{C}^*)^n$, that $\overline{\mathcal{M}}_{0,n}$ is the inverse limit of all GIT quotients $(\mathbb{P}^1)^n\quotient\mathtt{SL}_2$.  Assuming the $F$-conjecture (see the previous subsection) we can recognize any of the Hassett spaces appearing in Theorem \ref{thm:contraction} as (the normalization of) the inverse limit of an appropriate family of $\mathtt{SL}_3$ GIT quotients.

\begin{proposition}
Fix $\vec{c}\in (0,1]^n\cap\mathbb{Q}^n$ with $2 < c < 3$ and write $\vec{c'} \le \vec{c}$ if $c'_i \le c_i$ for $i=1,\ldots,n$.  Assuming the $F$-conjecture, the normalization of the inverse limit of all quotients $\mathtt{Con}(n)\quotient_{(\gamma',\vec{c'})}\mathtt{SL}_3$ such that $\vec{c'} \le \vec{c}$ and $2 < \gamma' = 3 - c'$ is $\overline{\mathcal{M}}_{0,\vec{c}}$.
\end{proposition}

\begin{remark}
We do not need the full $F$-conjecture to prove this, only a weaker form which asserts that $F$-curves generate the cone of curves for Hassett spaces $\overline{\mathcal{M}}_{0,\vec{c}}$ satisfying $2 < c < 3$.  In (\cite{Simp}, Theorem 3.3.2) this ``weak'' $F$-conjecture is proven for symmetric weights in this range, so in such cases the above proposition is unconditional.  
\end{remark}

\begin{proof}
By Theorem \ref{thm:contraction} there is a birational morphism $\overline{\mathcal{M}}_{0,\vec{c'}} \rightarrow \mathtt{Con}(n)\quotient_{(\gamma',\vec{c'})}\mathtt{SL}_3$ for any $\vec{c'} \le \vec{c}$, so by composing with the reduction morphisms $\overline{\mathcal{M}}_{0,\vec{c}} \rightarrow \overline{\mathcal{M}}_{0,\vec{c'}}$ we get a birational morphism from $\overline{\mathcal{M}}_{0,\vec{c'}}$ to any of the GIT quotients described in the proposition---and hence to their inverse limit $\mathcal{L}$.  This inverse limit may not be normal, but since all the Hassett spaces are normal this induced morphism factors through the normalization $\widetilde{\mathcal{L}} \rightarrow \mathcal{L}$.  We claim the morphism $\overline{\mathcal{M}}_{0,\vec{c}} \rightarrow \widetilde{\mathcal{L}}$ is an isomorphism.

If it were not then it would have to contract a curve, and because we are assuming the $F$-conjecture this means it would have to contract an $F$-curve.  This $F$-curve corresponds to a partition $N_1\cup\cdots\cup N_4 = \{1,\ldots,n\}$ and as above we write $x_k = \sum_{i\in N_k} c_i$ and assume without loss of generality $x_1 \le \cdots \le x_4$.  For this $F$-curve to be contracted by $\overline{\mathcal{M}}_{0,\vec{c}} \rightarrow \widetilde{\mathcal{L}}$ it must not be contracted by $\overline{\mathcal{M}}_{0,n} \rightarrow \overline{\mathcal{M}}_{0,\vec{c}}$, so $x_1 + x_2 + x_3 > 1$.  Now the morphisms $\overline{\mathcal{M}}_{0,\vec{c}} \rightarrow \mathtt{Con}(n)\quotient_{(\gamma',\vec{c'})}\mathtt{SL}_3$ all factor through $\widetilde{\mathcal{L}}$ so they must \emph{all} contract this same $F$-curve.  By Lemma \ref{F-curves} this means in particular, taking $\vec{c'}=\vec{c}$, that $x_3 > 1$ (and hence $x_2 < 1$ since $\sum_{k=1}^4 x_k < 3$).  Choose $\vec{c'}\le\vec{c}$ such that $x'_3 = 1$, where $x'_k := \sum_{i\in N_k}c'_i$.  We still have $x'_1 \le x'_2 \le x'_3 \le x'_4$ since $x'_2 = x_2 < 1 = x'_3 < x_3 \le x_4 = x'_4$, and we certainly have $x'_1+x'_2+x'_3 > 1$ since $x'_3 =1$ and $x'_1 > 0$, and of course $\sum_{k=1}^4 x'_k > 2$, so by Lemma \ref{F-curves} this $F$-curve is \emph{not} contracted by $\overline{\mathcal{M}}_{0,\vec{c}} \rightarrow \mathtt{Con}(n)\quotient_{(\gamma',\vec{c'})}\mathtt{SL}_3$, a contradiction.  Therefore $\overline{\mathcal{M}}_{0,\vec{c}} \rightarrow \widetilde{\mathcal{L}}$ must in fact be an isomorphism.
\end{proof}

\subsection{GIT Cones}

If $X$ is a projective variety acted upon by a reductive group $G$, and $L$ is any $G$-linearized ample line bundle on $X$, then the GIT quotient may be defined as $X\quotient_L G := \mathtt{Proj}(\oplus_{m\ge 0}\Gamma(X,L^{\otimes m})^G)$.  This perspective makes it clear that the quotient comes equipped with a distinguished polarization, i.e. an ample line bundle $L'$.  If $\phi : Y \rightarrow X\quotient_L G$ is any morphism then the pull-backed line bundle $\phi^*L'$ on $Y$ is nef.  Understanding the nef cone of a variety is a crucial part of understanding its birational geometry, so Alexeev and Swinarski (\cite{AS}) used this GIT setup to study the nef cone of $\overline{\mathcal{M}}_{0,n}$.  

Specifically, each of the quotients $(\mathbb{P}^1)^n\quotient_L\mathtt{SL}_2$ comes with a distinguished polarization $L'$ determined by $L\in\Delta(2,n)$ so pulling back along the Kapranov morphisms $\overline{\mathcal{M}}_{0,n} \rightarrow (\mathbb{P}^1)^n\quotient_L\mathtt{SL}_2$ produces a collection of nef line bundles on $\overline{\mathcal{M}}_{0,n}$ which Alexeev and Swinarski term the \emph{GIT cone}.  For $n=5$ this GIT cone coincides with the full nef cone; for $n \ge 6$ it is strictly smaller, though still a useful object.

Theorem \ref{thm:contraction} allows one to setup an analogous framework for $\mathtt{SL}_3$ quotients.  That is, as one varies the linearization $L\in\Delta(3,n+1)$ one can pull back the distinguished polarization on $\mathtt{Con}(n)\quotient_L\mathtt{SL}_3$ along the morphism $\overline{\mathcal{M}}_{0,n} \rightarrow \mathtt{Con}(n)\quotient_L\mathtt{SL}_3$, thereby producing a collection of nef line bundles on $\overline{\mathcal{M}}_{0,n}$ which we term the \emph{second order GIT cone}.  The set of line bundles coming from linearizations on the face $\gamma = 1$ of $\Delta(3,n+1)$ coincide with the first order GIT cone defined in \cite{AS}, so by allowing arbitrary $\gamma$ we except to find a strictly larger GIT cone.  This has not been explored much yet; we leave it as an open area of study.

\end{document}